\documentclass[11pt, twoside]{article}

\usepackage{graphicx}
\usepackage[caption=false]{subfig}
\usepackage{amsmath}
\usepackage{amssymb,amsfonts}
\usepackage{amsthm}
\usepackage{bm}
\usepackage{mathrsfs}
\usepackage{amssymb}
\usepackage{multirow}

\usepackage[margin=1in]{geometry}

\usepackage{algorithm}
\usepackage{algorithmic}

\numberwithin{equation}{section}
\usepackage{multirow}
\usepackage{makecell}
\usepackage{booktabs}
\theoremstyle{definition}
\newtheorem{theorem}{Theorem}

\newtheorem{lemma}{Lemma}
\newtheorem{corollary}{Corollary}

\newtheorem{remark}{Remark}
\newtheorem{example}{Example}

\usepackage{cite}
\usepackage{hyperref}
\usepackage[nameinlink]{cleveref}
\newcommand{\vertiii}[1]{{\left\vert\kern-0.25ex\left\vert\kern-0.25ex\left\vert #1 
		\right\vert\kern-0.25ex\right\vert\kern-0.25ex\right\vert}}

\usepackage{fancyhdr}
\begin{document}
	
	\title{Sharp $L^\infty$ estimates of HDG methods for Poisson equation II: 3D}

\author{Gang Chen%
	\thanks{ College of  Mathematics, Sichuan University, Chengdu 610064, China (\mbox{cglwdm@uestc.edu.cn}).}
	\and
	Peter Monk%
	\thanks{Department of Mathematics Science, University of Delaware, Newark, DE, USA (\mbox{monk@udel.edu}).}
	\and
	Yangwen Zhang%
	\thanks{Department of Mathematics Science, Carnegie Mellon University, Pittsburgh, PA, USA (\mbox{yangwenz@andrew.cmu.edu}).}
}

\date{\today}

\maketitle

\begin{abstract}
 In [SIAM J. Numer. Anal., 59 (2), 720-745], we proved quasi-optimal $L^\infty$ estimates (up to logarithmic factors) for the solution of Poisson's equation  by a hybridizable discontinuous Galerkin (HDG) method. However, the  estimates \emph{only} work  in 2D. In this paper, we use the approach in [Numer. Math., 131 (2015), pp. 771–822] and  obtain \emph{sharp} (without logarithmic factors) $L^\infty$ estimates for the HDG method in both 2D and 3D. Numerical experiments are presented to confirm our theoretical result.
\end{abstract}

\section{Introduction}
This is the second in a series of papers devoted to proving  $L^\infty$ norm estimates for the hybridizable discontinuous Galerkin (HDG) method  applied to elliptic partial differential equations (PDEs). 	The HDG methods were proposed by Cockburn et al. \cite{Cockburn_Gopalakrishnan_Lazarov_Unify_SINUM_2009}, have the same advantages as typical DG methods but have many less globally coupled unknowns.  HDG methods are currently undergoing rapid development and have been used in many applications; see, e.g., \cite{monk2021hdg,MR3992054,HuShenSinglerZhangZheng_HDG_Dirichlet_control1,ChenCockburnSinglerZhang1,MR3987425,Chen_Monk_Peter1}.

Let us place our results within the ongoing effort of proving sharp $L^\infty$ estimates for PDEs. The first technique for $L^\infty$ norm
estimation  was developed in series of papers by Schatz and Wahlbin \cite{Schatz_Math_Comp_1977,Schatz_Math_Comp_1982,Schatz_Math_Comp_1995}. They used dyadic decomposition of the domain and require local energy estimates together with sharp pointwise estimates for the corresponding components of the Green’s matrix. For smooth domains this technique was successfully used in \cite{Chen_SINUM_2006} for mixed methods,  in \cite{Guzman_Math_Comp_2008,MR1954085} for discontinuous Galerkin (DG) methods and in \cite{Chen_Math_Comp_2005} for  local DG methods. This technique was also applied  for the  Stokes equations, see Guzm\'{a}n  and Leykekhman \cite{Guzman_Math_Comp_2012}. 

Another technique is based on weighted $L^2$ norms. In 1976, Scott \cite{Scott_Math_Comp_1976} proved a quasi-optimal $ L^\infty $ norm estimates for the conforming finite element method. Later on, Gastaldi and Nochetto \cite{Gastaldi_M2AN_1989} extended this technique for mixed methods; see also \cite{Scholz_M2AN_1977,Gastaldi_NM_1987,Duran_M2AN_1988,Duran_M2AN_1988_nonlinear,Gastaldi_M2AN_1989,wang1989asymptotic}.  Since there is a strong relation between the HDG and the mixed methods (see \cite{Cockburn_Gopalakrishnan_Lazarov_Unify_SINUM_2009}), it is reasonable to ask if similar estimates could be obtained for the HDG method.

In part I of this work \cite{ChenMonkZhangLinfty1SINUM2021}, we considered  $L^\infty$ estimates for the HDG approximation of the solution of the  following elliptic system:
\begin{subequations}\label{Poisson}
	\begin{align}
		\bm q + \nabla u &= \bm 0 \quad  \text{in} \; \Omega \label{p1},\\
		\nabla \cdot \bm q  &=  f \quad  \text{in} \; \Omega \label{p2},\\
		u &= 0\quad  \text{on}\;  \partial\Omega, \label{p3}
	\end{align}
\end{subequations}
where $\Omega\subset \mathbb R^2$ is a convex polygonal Lipschitz domain with  boundary $\partial \Omega$ and $f\in L^2(\Omega)$. We proved quasi-optimal $L^{\infty}$ error estimates for the HDG approximation to $u$ and $\bm q$.  However, the estimates are only valid for a triangular mesh  in 2D and polynomials of  degree $k\ge 1$. The current paper is devoted mainly to $\mathbb R^3$, in which case $\Omega$ is a convex, polyhedral,  Lipschitz domain. 

The roadblock for  optimal $L^\infty$ error estimates for the HDG method in 3D   is that the HDG projection (see \cite[Proposition 2.1]{Cockburn_Gopalakrishnan_Sayas_Porjection_MathComp_2010}) only has a weak commutative property. Hence, we use different finite element spaces and numerical fluxes  such that the corresponding HDG projection satisfies  a strong commutative property. We note that this choice was first proposed by Lehrenfeld in \cite{Lehrenfeld_PhD_thesis_2010}. However, this brings new challenges, we  are not only need to prove an optimal weighted $L^2$ estimates for the flux of the Green's function, but also the Green's function itself. Moreover, we use the approach in \cite{Scott_NM_2015} to remove the logarithmic factors in the error analysis; see \Cref{Linftyestimates}. To the best of our knowledge, this is the first such result for mixed methods and HDG methods. Finally,  we present  numerical experiments in \Cref{NumericalExperiments}  to confirm our theoretical results from \Cref{main_result_Linfty_norm}.

\section{HDG formulation and preliminary material}
Throughout the paper we adopt the standard notation $W^{m,p}(D)$ for Sobolev spaces on a bounded domain $D\subset \mathbb R^d$  ($d=2,3$) with norm $\|\cdot\|_{W^{m,p}(D)}$ and seminorm $|\cdot|_{W^{m,p}(D)}$:
\begin{align*}
	\|u\|_{W^{m,p}(D)}^p = \sum_{|i|\le m}\int_{D} |D^i u|^p {\rm d} \bm x,\quad
	|u|_{W^{m,p}(D)}^p = \sum_{|i|= m}\int_{D} |D^i u|^p {\rm d} \bm x,
\end{align*}
where $i$ is a multi-index and $D^i$ is the corresponding partial differential operator of order $|i|$. We denote $W^{m,2}(D)$ by $H^{m}(D)$ with norm $\|\cdot\|_{H^m(D)}$ and seminorm $|\cdot|_{H^m(D)}$. Specifically, $H_0^1(D)=\{v\in H^1(D):v=0 \;\mbox{on}\; \partial D\}$.  We denote the $L^2$-inner products on $D$ and $S$ by
\begin{align*}
	(v,w)_{D} = \int_{D} vw \;{\rm d}\bm x \quad \forall v,w\in L^2(D),\quad 
	\left\langle v,w\right\rangle_{S} = \int_{S} vw \;{\rm d}\bm x\quad\forall v,w\in L^2(S),
\end{align*}
where $D\subset \mathbb R^d$ and $S$ is a surface in $\mathbb R^{d-1}$.  Finally, we define the space $\bm H(\text{div},\Omega) $ as usual
\begin{align*}
	\bm H(\text{div},\Omega) = \{\bm{v}\in [L^2(\Omega)]^d\,:\, \nabla\cdot \bm{v}\in L^2(\Omega)\}.
\end{align*}

{Let $\mathcal{T}_{h}$ be a collection of disjoint polyhedral elements $K$ that partition $\Omega$. We  assume that all the elements are shape-regular and quasi-uniform in the sense of \cite{dong2021hybrid}}. We denote by $\partial \mathcal{T}_h$ the set $\{\partial K: K\in \mathcal{T}_h\}$. For an element $K$ of the mesh  $\mathcal{T}_h$, let $F = \partial K \cap \partial\Omega$ denotes the boundary face of $ K $ having non-zero $d-1$ dimensional Lebesgue measure. Let $\mathcal F_h^\partial$ be the set of boundary faces and  $\mathcal F_h$ denote the set of all faces. We define the following mesh dependent norms  and spaces by
\begin{gather*}
	(w,v)_{\mathcal{T}_h} = \sum_{K\in\mathcal{T}_h} (w,v)_K,   \quad\quad\quad\quad\left\langle \zeta,\rho\right\rangle_{\partial\mathcal{T}_h} = \sum_{K\in\mathcal{T}_h} \left\langle \zeta,\rho\right\rangle_{\partial K},\\
	H^1(\mathcal T_h) = \prod_{K\in \mathcal T_h} H^1(K), \quad\quad\quad\quad  L^2 (\partial\mathcal T_h) =  \prod_{K\in \mathcal T_h} L^2(\partial K).
\end{gather*}

Let $\mathcal{P}^k(D)$ (resp $\mathcal{P}^k(S)$) denote the set of polynomials of degree at most $k$ on a domain $D\subset \mathbb R^d$ (resp.  a plane $S\subset \mathbb R^d$).  We introduce the discontinuous finite element spaces used in the HDG method as follows:
\begin{alignat*}{4}
	\bm{V}_h  &:= \{\bm{v}\in [L^2(\Omega)]^d:&&\; \bm{v}|_{K}\in [\mathcal{P}^k(K)]^d, &&\;\forall \; K\in \mathcal{T}_h\},\\
	{W}_h  &:= \{{w}\in L^2(\Omega):&&\; {w}|_{K}\in \mathcal{P}^{k+1}(K),&&\; \forall \; K\in \mathcal{T}_h\},\\
	{M}_h  &:= \{{\mu}\in L^2(\mathcal{\mathcal{F}}_h):&&\; {\mu}|_F\in \mathcal{P}^k(F), &&\;\forall \; F\in \mathcal{F}_h, &&\;\mu|_{\mathcal{F}_h^\partial} = 0\}.
\end{alignat*}

\subsection{HDG  formulation}
In this paper, we will use the HDG scheme of \cite{Lehrenfeld_PhD_thesis_2010}. The HDG  method seeks the flux ${\bm{q}}_h \in \bm{V}_h $, the scalar variable $ u_h  \in W_h $ and  its numerical trace $ \widehat{u}_h\in \widehat W_h$ satisfying
\begin{subequations}\label{HDG_discrete2}
	\begin{align}
		(\bm{q}_h, \bm{v}_h)_{{\mathcal{T}_h}}- (u_h, \nabla\cdot \bm{v}_h)_{{\mathcal{T}_h}}+\langle \widehat u_h, \bm{v}_h\cdot \bm{n} \rangle_{\partial{{\mathcal{T}_h}}} &=0, \label{HDG_discrete2_a}\\
		-(\bm{q}_h, \nabla w_h)_{{\mathcal{T}_h}}
		+\langle\widehat{\bm{q}}_h\cdot\bm{n}, w_h \rangle_{\partial{{\mathcal{T}_h}}}&=(f, w_h)_{{\mathcal{T}_h}}, \label{HDG_discrete2_b}\\
		\langle\widehat{\bm{q}}_h\cdot \bm{n}, \widehat w_h \rangle_{\partial\mathcal{T}_{h}}&=0 \label{HDG_discrete2_e}
	\end{align}
	for all $(\bm{v}_h,w_h,\widehat w_h)\in \bm{V}_h\times W_h\times \widehat W_h$.
	The numerical {flux} on $\partial\mathcal{T}_h$ is defined by \cite{Lehrenfeld_PhD_thesis_2010}
	\begin{align}\label{num_tra}
		\widehat{\bm q}_h\cdot\bm n &= \bm q_h\cdot\bm n+h_K^{-1}(\Pi^{\partial}_ku_h-\widehat{u}_h),
	\end{align}
	where $\Pi_k^\partial$ is  the element-wise $L^2$ projection from $\mathcal P^{k+1}(F)$ to $\mathcal P^k(F)$:
	\begin{align}\label{L2_edge}
		\langle \Pi_k^\partial   w, \widehat w_h \rangle_F &= \langle w,  \widehat w_h  \rangle_F\qquad \forall \;   \widehat w_h \in \mathcal P^{k}(F).
	\end{align}
	This completes the definition of the  HDG formulation we shall analyze.
\end{subequations}

To shorten lengthy equations, we define the following HDG bilinear form $ \mathscr B: \left(\bm H^1(\mathcal T_h) \times\right.$ $\left. H^1(\mathcal T_h)\times L^2 (\partial\mathcal T_h) \right)  \times \left( \bm H^1(\mathcal T_h) \times H^1(\mathcal T_h)\times L^2 (\partial\mathcal T_h)\right)\to \mathbb R$ by
\begin{align}\label{def_B}
	\begin{split}
		\mathscr  B( \bm q,u,\widehat u;\bm v,w,\widehat w)&=(\bm{q}, \bm v)_{{\mathcal{T}_h}}- (u, \nabla\cdot \bm v)_{{\mathcal{T}_h}}+\langle \widehat{u}, \bm v\cdot \bm{n} \rangle_{\partial{{\mathcal{T}_h}}}\\
		&\quad- (\nabla\cdot\bm{q},  w)_{{\mathcal{T}_h}}+ \langle \bm q\cdot\bm n, \widehat w \rangle_{\partial \mathcal T_h}-\langle  h_K^{-1}(\Pi_k^{\partial} u - \widehat u),  \Pi_k^{\partial}w-\widehat w \rangle_{\partial{{\mathcal{T}_h}}},
	\end{split}
\end{align}
By the definition of $\mathscr  B$ in \eqref{def_B},  we can rewrite the HDG formulation of system \eqref{HDG_discrete2}, as follows: find $({\bm{q}}_h,u_h,\widehat u_h)\in \bm V_h\times W_h\times \widehat W_h$  such that
\begin{align}\label{HDG_full_discrete}
	\mathscr B (\bm q_h,u_h,\widehat u_h;\bm v_h,w_h,\widehat w_h) =  - (f,w_h)_{\mathcal T_h}
\end{align}
for all $\left(\bm{v}_h,w_h,\widehat w_h\right)\in \bm V_h\times W_h\times \widehat W_h$. Moreover, the exact solution $(\bm q, u)\in \bm H^1(\mathcal T_h) \times H^1(\mathcal T_h)$ satisfies equation \eqref{HDG_full_discrete}, i.e.,
\begin{align}\label{HDG_exact}
	\mathscr B (\bm q,u, u;\bm v_h,w_h,\widehat w_h) =  - (f,w_h)_{\mathcal T_h}.
\end{align}
The following lemma shows that the bilinear form $\mathscr B$ is symmetric and has an important positive property. It is proved by integration by parts. We do not provide details.
\begin{lemma}\label{eq_B}
	For any $ (\bm q,u,\widehat u; \bm v,w,\widehat w) \in \bm H^1(\mathcal T_h) \times H^1(\mathcal T_h)\times L^2 (\partial\mathcal T_h)\times  \bm H^1(\mathcal T_h) \times H^1(\mathcal T_h)\times L^2 (\partial\mathcal T_h)$, we have
	\begin{subequations}
		\begin{align}
			\mathscr B (\bm q,u,\widehat u;\bm v,w,\widehat w)  &=  	\mathscr B (\bm v,w,\widehat w;\bm q,u,\widehat u),\label{commute}\\
			\mathscr B (\bm q,u,\widehat u;\bm q,-u,-\widehat u) &= \|\bm q\|_{\mathcal T_h}^2 + \|h_K^{\frac 1 2}(\Pi_k^\partial u - \widehat u)\|_{\partial \mathcal T_h}^2.
		\end{align}			
	\end{subequations}
\end{lemma}

\subsection{Preliminary material}\label{sec:Projectionoperator}
We introduce the standard local $L^2$ projection operator $\Pi_\ell^o: L^2(K)\to \mathcal P^{\ell}(K)$  satisfying
\begin{align}\label{L2_do}
	(\Pi_\ell^o w, w_h)_K &= (w,w_h)_K\qquad \forall \; w_h\in \mathcal P^\ell(K).
\end{align}

We use $\bm \Pi_\ell^o$ to denote the local vector $L^2$ projection operator, the definition componentwise is the same as the  local scalar  $L^2$ projection operator.  The next lemma gives the approximation properties of $\Pi_\ell^o$ and its proof can be found in \cite[Theorem 3.3.3, Theorem 3.3.4]{book2}.
\begin{lemma}\label{approximation_of_L2}
	Let $\ell\ge 0$ be an integer and  $\rho\in [1, +\infty]$. If $(\ell+1)\rho<d$, then we require $d, \rho$ and $\ell$ to also satisfy $2\le \frac{d\rho}{d -(\ell+1)\rho }$. For $j\in \{0, 1 ,\ldots, \ell+1\}$, if $s_j$ satisfies
	\begin{align}\label{index_s}
		\begin{cases}
			\rho \le s_j \le  \frac{d\rho}{d -(\ell+1-j)\rho } &(\ell+1-j)\rho <d,\\
			\rho \le s_j  < \infty&(\ell+1-j)\rho =d,\\
			\rho \le s_j \le \infty &(\ell+1-j)\rho >d,
		\end{cases}
	\end{align}
	then there exists a constant $C$ which is independent of $K$ such that 
	\begin{subequations}
		\begin{align}
			\| \nabla^j(\Pi_\ell^o u - u) \|_{L^{s_j}(K)} &\le C h_K^{\ell+1-j+\frac{d}{s_j} - \frac{d}{\rho}}|u|_{W^{\ell+1, \rho}(K)},\label{Pro_jec_1}\\
			\| \nabla^j(\Pi_\ell^o u - u) \|_{L^{s_j}(\partial K)}&\le C h_K^{\ell+1-j+\frac{d-1}{s_j} - \frac{d}{\rho}}|u|_{W^{\ell+1, \rho}(K)}.\label{Pro_jec_2}
		\end{align}	
	\end{subequations}
\end{lemma}

In the analysis, we also need the following {standard} inverse inequality~{\cite[Theorem 3.4.1]{book2}}, that follows from our assumption of a shape regular mesh.
\begin{lemma}[Inverse inequality]\label{Inverse_inequality}
	{Let $k\ge 0$ be an integer, $\mu, \rho\in [1, +\infty]$, then there exists $C$ depend on $k, \mu, \rho$ and $d$ such that }
	\begin{align}\label{inverse}
		|v_h|_{t, \mu, K} \le Ch_K^{\frac{d}{\mu} - \frac{d}{\rho}-t+s}	|v_h|_{s,\rho, K}, \quad \forall v_h\in \mathcal P_k(K), \quad t\ge s.
	\end{align}
\end{lemma}

\subsection{Basic tools related to the  weighted function}
First,  we define the weight $\sigma_z$ by:
\begin{align} \label{def_sigma} 
	\sigma_z(\bm x) = (|\bm x - \bm z|^2 +    {\theta^2})^{\frac 1 2}.
\end{align}
where $\bm z\in \Omega$,  $\theta=\kappa h$ and $\kappa\ge 1$ { is a constant which will be discussed later}.

Next, we summarize some properties of the function which will be used later.

\begin{lemma}\cite{Scott_NM_2015}
	For any $\alpha\in \mathbb R$ there is a constant $C$ independent of $\alpha$ such that the function $\sigma_z$ has the following  properties:
	\begin{subequations}
		\begin{gather}
			\frac{\max_{\bm x\in K} \sigma_z(\bm x)^\alpha}{\min_{\bm x\in K} \sigma_z(\bm x)^\alpha}  \le 3^{\frac \alpha 2} \qquad \forall K\in \mathcal T_h, \label{pro_sigma_1}\\
			\left|\nabla^k (\sigma_z(\bm x)^\alpha)\right|\le C \sigma_z(\bm x)^{\alpha-k},\label{pro_sigma_2}\\
			\int_{\Omega} \sigma_z (\bm x)^{-d-\lambda} ~{\rm d}\bm x \le C\theta^{-\lambda} \qquad\forall 0<\lambda<1,\label{sigma_lemma_2_constant}\\
			\|\sigma_z^{\alpha}\|_{\partial\mathcal T_h}\le C
			\theta^{-\frac 1 2}\|\sigma_z^{\alpha}\|_{\mathcal T_h}.
			\label{sigma-bd}
		\end{gather}
	\end{subequations}
\end{lemma}

\begin{proof}
	The proof of \eqref{pro_sigma_1}, \eqref{pro_sigma_2} and  \eqref{sigma_lemma_2_constant} can be found in \cite[Lemma 2.1, equation (2.2)]{Scott_JMPA_2005} and \cite[equation (1.42)]{Scott_NM_2015}, respectively.
	For \eqref{sigma-bd} we have 
	\begin{align*}
		\|\sigma_z^{\alpha}\|_{L^2(\partial K)}^2\le C 	\|\sigma_z^{\alpha}\|_{L^2(K)}	\|\nabla(\sigma_z^{\alpha})\|_{\bm L^2(K)}\le C	\|\sigma_z^{\alpha}\|_{L^2(K)}	\|\sigma_z^{\alpha-1}\|_{ L^2(K)} \le C\theta^{-\frac 1 2}	\|\sigma_z^{\alpha}\|_{L^2(K)}^2.
	\end{align*}	
	Now \eqref{sigma-bd}  follows by adding over all the elements.
\end{proof}

The next lemma provides weighted norm error estimates for the local $L^2$ projection.

\begin{lemma}\label{L2_noundness_sigma} 
	For any integer $j\ge 0$, we have the weighted approximation	
	\begin{subequations}
		\begin{align}\label{L2_prijection_sigma}
			\left\|\sigma_z^{\alpha}(  v-  \Pi_j^{o}   v)\right\|_{\mathcal T_h}  + 	h_K^{\frac 1 2}\left\|\sigma_z^{\alpha}(  v-  \Pi_j^{o}   v)\right\|_{ \partial\mathcal T_h}  \le Ch^{j+1} \|\sigma_z^{\alpha} \nabla^{j+1}   v\|_{\mathcal T_h},
		\end{align}
		if  $v_h|_K\in \mathcal P^j(K)$, then we have  the  following weighted superconvergence
		\begin{align}\label{weight-app}
			\begin{split}
				\hspace{2em}&\hspace{-2em}	\left\|\sigma_z^{-{\frac \alpha 2}}\left(\sigma_z^{{\alpha}} v_h-\Pi_j^{o}(\sigma_z^{{\alpha}} v_h)\right)\right\|_{\mathcal T_h}   +	h^{\frac 1 2}\left\|\sigma_z^{-{\frac \alpha 2}}\left(\sigma_z^{{\alpha}} v_h-\Pi_j^{o}(\sigma_z^{{\alpha}} v_h)\right)\right\|_{\partial\mathcal T_h} \\
				& + h^{\frac 1 2}	\left\|\sigma_z^{-{\frac \alpha 2}}\left(\sigma_z^{{\alpha}} v_h-\Pi_j^{\partial}(\sigma_z^{{\alpha}} v_h)\right)\right\|_{\partial\mathcal T_h} \le Ch\left\|\sigma_z^{{\frac \alpha 2}-1} v_h\right\|_{\mathcal T_h}. 
			\end{split}
		\end{align}
	\end{subequations}
\end{lemma}
The proof of \Cref{L2_noundness_sigma} can be found in \cite[Lemma 3.6]{ChenMonkZhangLinfty1SINUM2021}. The idea behind the proof is to use \eqref{pro_sigma_1} and \emph{local} estimates on each element, and then sum over all elements.  Using the same technique, we can prove the following weighted estimates, and the details are omitted.

\begin{lemma}[Weighted inverse inequality] We have
	\begin{align}\label{weighed-inverse2}
		\|\sigma_z^{\alpha}v_h\|_{\partial\mathcal T_h}\le Ch^{-\frac{1}{2}}\|\sigma_z^{\alpha}v_h\|_{\mathcal T_h}.
	\end{align}
\end{lemma}

\begin{lemma}[Weighted Oswald interpolation]\cite{MR2034620}\label{Oswarld_interpolation}
	There exists an interpolation operator $\mathcal I_h: W_h\to W_h\cap H_0^1(\Omega)$, such that 
	\begin{align}\label{Oswarld_interpolation_b}
		\|\sigma_z^{\alpha} (\mathcal I_hv_h-v_h)\|_{\mathcal T_h} + h	\|\sigma_z^{\alpha}\nabla (\mathcal I_hv_h-v_h)\|_{\mathcal T_h}&\le C\|\sigma_z^{\alpha}h^{\frac{1}{2}}[\![v_h]\!]\|_{\mathcal E_h}.
	\end{align}
\end{lemma}

\section{$L^\infty$ norm estimates}\label{Linftyestimates}
\subsection{Main result}
Now, we  state the main result of our paper:
\begin{theorem}\label{main_result_Linfty_norm}
	Let $(\bm q, u)$ and $(\bm q_h, u_h,\widehat u_h)$ be the solution of \eqref{Poisson} and \eqref{HDG_discrete2}, respectively.  First, if  $u\in L^\infty(\Omega)$ and  $\bm q\in \bm L^\infty(\Omega)$, then we have the following stability bounds:
	\begin{subequations}\label{maxmimu_norm}
		\begin{align}
			\|u_h\|_{L^\infty(\Omega)} &\le   C(\|u\|_{L^\infty(\Omega)}   + h\|\bm q\|_{\bm L^\infty(\Omega)}),\label{maxmimu_norm_u} \\
			\|\bm q_h\|_{L^\infty(\Omega)} &\le  C \|\bm q\|_{\bm L^\infty(\Omega)}. \label{maxmimu_norm_q} 
		\end{align}	
		Second, if $(\bm q, u)\in \bm W^{k+1,\infty}(\Omega) \times  W^{k+2,\infty}(\Omega)$,  we have
		\begin{align}
			\|\bm q- \bm q_h\|_{\bm L^\infty(\Omega)} &\le Ch^{k+1}(|\bm q|_{\bm W^{k+1,\infty}(\Omega)} + |u|_{ W^{k+2,\infty}(\Omega)}),\label{thmL_infty_q}\\
			\|u- u_h\|_{L^\infty(\Omega)} &\le Ch^{k+2}(|\bm q|_{\bm W^{k+1,\infty}(\Omega)} + |u|_{ W^{k+2,\infty}(\Omega)})\label{thmL_infty_u}.
		\end{align}
	\end{subequations}
\end{theorem}

\begin{remark}\label{pre}
	In \cite{ChenMonkZhangLinfty1SINUM2021}, we obtained quasi-optimal $L^\infty$ norm error estimates, but the domain is restricted to  two dimensional space, the mesh  has to  be triangular and the polynomial degree $k\ge 1$. The result in \Cref{main_result_Linfty_norm} (which holds in $\mathbb R^2$ or $\mathbb R^3$) relaxes all the above constraints. It is worthwhile mentioning that the constant $C$ in \eqref{main_result_Linfty_norm} does not depend on logarithmic factors $(\log h)$, which is the \emph{first} such result for mixed methods  in the literature.
\end{remark}

In \cite{ChenMonkZhangLinfty1SINUM2021} we showed that a useful application of $L^\infty$ estimates is  to prove flux estimates on interfaces in the mesh. We give the corresponding result next.
\begin{corollary}\label{optimal_es_boundary}
	Let $\Gamma$ be a finite union of line segments in 2D or plan surface in 3D  such that $\Omega$ is decomposed into finitely many Lipschitz domains by $\Gamma$.  Define $\mathcal F_h^\Gamma$ by
	\begin{align*}
		\mathcal F_h^\Gamma = \{F\in \mathcal F_h\,:\,\textup{measure} (F\cap \Gamma)>0\}.
	\end{align*}
	We assume $\Gamma$ can be written as the union of $\mathcal O(h^{1-d})$ edges or faces in $\mathcal F_h$, i.e., $\bar \Gamma = \bigcup_{F\in\mathcal F_h^\Gamma\subset \mathcal F_h} \bar F$.  If the  assumptions in \Cref{main_result_Linfty_norm} hold, then we have:
	\begin{subequations}
		\begin{align}
			\| \bm q - \bm q_h \|_{\bm L^2(\Gamma)} &\le Ch^{k+1}(|\bm q|_{\bm W^{k+1,\infty}(\Omega)} + |u|_{ W^{k+2,\infty}(\Omega)}),\label{L_infty_q}\\
			\| u- u_h \|_{L^2(\Gamma)} &\le Ch^{k+2}(|\bm q|_{\bm W^{k+1,\infty}(\Omega)} + |u|_{ W^{k+2,\infty}(\Omega)}). \label{L_infty_u}
		\end{align}
	\end{subequations}	
\end{corollary}
The proof of \eqref{optimal_es_boundary} is based on the proof in \cite[Theorem 4.1]{ChenMonkZhangLinfty1SINUM2021} and \Cref{main_result_Linfty_norm}.

The remainder of this section will be devoted to proving our main result, \Cref{main_result_Linfty_norm}.

\subsection{Proof of \Cref{main_result_Linfty_norm}} 
{Before starting  the proof of \Cref{main_result_Linfty_norm} we recall the definition of suitable regularized Green's functions.
	We follow the notation of Girault, Nochetto and Scott \cite{Scott_NM_2015}. 
	Let  $\varphi_h$ be a polynomial in $\mathcal P^k$ on each element,  $\bm x_M$ be the  point such that $|\varphi(\bm x_M)| = \max_{\bm x\in \bar\Omega}|\varphi(\bm x)|$,  $K$ be an element containing $\bm x_M$ and  $B \subset K$ be the {disk of} radius $\rho_K$ inscribed in $K$. Then there exists a smooth function $\delta_{M}\in C_0^{\infty}(B)$ supported in $B$ such that
	\begin{align*}
		\int_{\Omega} \delta_{M} ~{\rm d}\bm x &=1 \\
		\left\|\varphi_h\right\|_{L^{\infty}(\Omega)} &=\left|\int_{\Omega} \delta_{M} \varphi_h ~{\rm d}\bm x \right|,	
	\end{align*}
	and
	\begin{align}\label{Bound_L1}
		\left\|\delta_{M}\right\|_{L^{t}(B)} &\leq \frac{C_{t}}{\rho_{K}^{d\left(1-\frac{1}{t}\right)}},
	\end{align}
	for any number $t$ with $1 \leq t \leq \infty$, where the constant $C_{t}$ independent of mesh size $h$.  Here we interpret $\dfrac{1}{t}=0$ in the case $t=\infty$.}
%
%
%
%
%

The main idea behind the proof of $L^\infty$ norm estimates is to use the above defined smooth $\delta_{M}$ function. Given a scalar function $\delta_{1,z}$ and a vector $\bm \delta_{2,z}$ of the above type,  we define two regularized Green's functions for problem \eqref{Poisson} in mixed form:
\begin{align}\label{Green1}
	\bm \Phi_1+\nabla \Psi_1=\bm 0 \; \text{in}\ \Omega, \quad \nabla\cdot\bm  \Phi_1= \delta_{1,z}\;\text{in}\ \Omega,\quad \Psi_1 = 0\; \text{on}\ \partial\Omega,
\end{align}
and
\begin{align}\label{Green2}
	\bm \Phi_2+\nabla \Psi_2= \bm \delta_{2,z}\; \text{in}\ \Omega, \quad \nabla\cdot\bm  \Phi_2= 0\; \text{in}\ \Omega, \quad \Psi_2 = 0\; \text{on}\ \partial\Omega.
\end{align}

{We need two auxiliary results before starting the proof of  \Cref{main_result_Linfty_norm}. The proof of  \Cref{regularity_Green_function} can be found in \cite{Scott_JMPA_2005}.
	\begin{lemma}[Regularity for $\Psi_1$ and  $\Psi_2$]\label{regularity_Green_function} Suppose $z\in \Omega$, let $0<\lambda <1$ and $\mu=d+\lambda$.
		Let $\Psi_1$ and $\Psi_2$ be the solution of \eqref{Green1} and \eqref{Green2}, respectively. Then we have:
		\begin{subequations}
			\begin{align}
				&\|\Psi_1\|_{H^2(\Omega)}\le Ch^{-\frac d 2},&\|\sigma_z^{\frac{\mu}{2}-1} D \Psi_1\|_{L^2(\Omega)} +\|\sigma_z^{\frac{\mu}{2}} D^2 \Psi_1\|_{L^2(\Omega)}\le Ch^{\frac{\lambda}{2}},\label{regularity_psi_1}\\
				&\|\Psi_2\|_{H^2(\Omega)}\le Ch^{-\frac d 2-1},&\|\sigma_z^{\frac{\mu}{2}-1} D \Psi_2\|_{L^2(\Omega)} +\|\sigma_z^{\frac{\mu}{2}} D^2 \Psi_2\|_{L^2(\Omega)}\le Ch^{\frac{\lambda}{2}-1}.\label{regularity_psi_2}
			\end{align}
		\end{subequations}
	\end{lemma}

	We split the proof of \Cref{main_result_Linfty_norm} into four steps. First, we give standard $L^2$ estimates of the solutions  of \eqref{Green1} and  \eqref{Green2}. Second,  we shall obtain weighted $L^2$ norm approximations. Third, we prove the  $L^\infty$ norm stability of $\bm q_h$ and $u_h$. Finally, we obtain $L^\infty$ norm error estimates of $\bm q-\bm q_h$ and $u - u_h$. 
	
	\paragraph*{Step 1: $L^2$ norm error estimates for  the regularized Green's functions}																			
	Let $(\bm \Phi_{1,h}, \Psi_{1,h},\widehat \Psi_{1,h})$ and $(\bm \Phi_{2,h}, \Psi_{2,h}, \widehat \Psi_{2,h})$ be the HDG solution of \eqref{Green1} and \eqref{Green2}, respectively, i.e.,
	\begin{subequations}
		\begin{align}
			\mathscr B(\bm \Phi_{1,h}, \Psi_{1,h}, \widehat{\Psi}_{1,h}; \bm v_h, w_h,\widehat w_h) & = -(\delta_{1,z}, w_h)_{\mathcal T_h},\label{Green_HDG_1}\\
			\mathscr B(\bm \Phi_{2,h}, \Psi_{2,h}, \widehat{\Psi}_{2,h}; \bm v_h, w_h,\widehat w_h) & = ( \bm \delta_{2,z}, \bm v_h)_{\mathcal T_h}\label{Green_HDG_2}
		\end{align}
		for all $(\bm v_h, w_h, \widehat w_h)\in\bm{V}_h\times W_h\times \widehat W_h$. 
	\end{subequations}
	Through the paper, we use the following notation: 
	\begin{align*}
		\mathcal E_h^{\bm{\Phi}_i} = \bm \Pi_{k}^{o} \bm \Phi_i - \bm \Phi_{i,h}, \quad  \mathcal E_h^{\Psi_i} =  \Pi_{k+1}^{o}  \Psi_i -  \Psi_{i,h},  \quad \mathcal E_h^{\widehat\Psi_i} = \Pi_k^\partial   \Psi_i - \widehat  \Psi_{i,h}, i=1,2.
	\end{align*}
	
	Next, we list some preliminary results below, the proof can be found in \cite[Section 4.3]{HuShenSinglerZhangZheng_HDG_Dirichlet_control1}.
	\begin{subequations}
		\begin{align}
			\mathscr B(\bm \Pi_{k}^{o} \bm \Phi_1, \Pi_{k+1}^{o}  \Psi_1, \Pi_k^\partial   \Psi_1; \bm v_h, w_h,\widehat w_h) & = 
			\mathscr E(\bm{\Phi}_1,\Psi_1;\bm v_h,w_h,\widehat w_h)
			- (\delta_{1,z}, w_h)_{\mathcal T_h},\label{pro_error_equa_Phi}\\
			\mathscr B(\bm \Pi_{k}^{o} \bm \Phi_2, \Pi_{k+1}^{o}  \Psi_2, \Pi_k^\partial   \Psi_2; \bm v_h, w_h,\widehat w_h) & = \mathscr E(\bm{\Phi}_2,\Psi_2;\bm v_h,w_h,\widehat w_h)+(\bm \delta_{2,z}, \bm v_h)_{\mathcal T_h}, \label{pro_error_equa_Psi}\\
			\mathscr B(\mathcal E_h^{\bm{\Phi}_1},\mathcal E_h^{\Psi_1},\mathcal E_h^{\widehat\Psi_1}; \bm v_h, w_h,\widehat w_h) & =
			\mathscr E(\bm{\Phi}_1,\Psi_1;\bm v_h,w_h,\widehat w_h),\label{error_equa_Phi}\\
			\mathscr B(\mathcal E_h^{\bm{\Phi}_2},\mathcal E_h^{\Psi_2},\mathcal E_h^{\widehat\Psi_2}; \bm v_h, w_h,\widehat w_h) & = \mathscr E(\bm{\Phi}_2,\Psi_2;\bm v_h,w_h,\widehat w_h),\label{error_equa_Psi}
		\end{align}
	\end{subequations}
	where \begin{align*}
		\mathscr E(\bm{\Phi},\Psi;\bm v_h,w_h,\widehat w_h)  =\langle(\bm \Phi -\bm \Pi_{k}^{o} \bm\Phi )\cdot\bm n,w_h-\widehat{w}_h\rangle_{\partial\mathcal T_h} +\langle h_K^{-1} (\Psi -\Pi_{k+1}^o  \Psi   ),\Pi_k^\partial  w_h-\widehat w_h \rangle_{\partial\mathcal T_h}.
	\end{align*}
	Using \Cref{eq_B} and standard HDG error analysis, we have the following basic  $L^2$ estimates:
	\begin{subequations}
		\begin{align}
			\|\mathcal E_h^{\bm{\Phi}_1}\|_{\bm L^2(\Omega)}+\|h_K^{-\frac 1 2}(\Pi_k^{\partial}\mathcal E_h^{\Psi_1}-\mathcal E_h^{\widehat{\Psi}_1})\|_{\partial\mathcal T_h}&\le Ch^{1-\frac{d}{2}},
			&\|\mathcal E_h^{\Psi_1}\|_{L^2(\Omega)}\le  Ch^{2-\frac{d}{2}},\label{L2_error_Phi}\\
			\|\mathcal E_h^{\bm{\Phi}_2}\|_{\bm L^2(\Omega)}+\|h_K^{-\frac 1 2}(\Pi_k^{\partial}\mathcal E_h^{\Psi_2}-\mathcal E_h^{\widehat{\Psi}_2})\|_{\partial\mathcal T_h}&\le Ch^{-\frac{d}{2}},
			&\|\mathcal E_h^{\Psi_2}\|_{L^2(\Omega)}\le  Ch^{1-\frac{d}{2}},\label{L2_error_Psi}
		\end{align}
		and the following discrete  inequality  on each element $K\in \mathcal T_h$, $i=1,2$.
		\begin{align}\label{psi-phi}
			\|\nabla \mathcal E_h^{\Psi_i}\|_{\bm L^2(K)} + \|{h_K^{-\frac 1 2}}(\mathcal E_h^{\Psi_i}-\mathcal E_h^{\widehat{\Psi}_i})\|_{L^2(\partial K)}\le C(
			\|\mathcal E_h^{\bm{\Phi}_i}\|_{\bm L^2(K)}+\|{h_K^{-\frac 1 2}}(\Pi_k^{\partial}\mathcal E_h^{\Psi_i}-\mathcal E_h^{\widehat{\Psi}_i})\|_{L^2(\partial K)}).
		\end{align}
	\end{subequations}

	\paragraph*{Step 2: Weighted $L^2$ norm error estimates for  the regularized Green's functions}	
	First, we give the weighted $L^2$ norm  estimate for the stabilization term. Since estimate \eqref{psi-phi} holds on each element, we use the same technique as in \Cref{L2_noundness_sigma} to get
	\begin{align}\label{Stbilization_weighted}
		\begin{split}
			\|\sigma_z^{\frac{\mu}{2}}\nabla \mathcal E_h^{\Psi_i}\|_{\mathcal T_h}& +\|h_K^{-\frac 1 2}\sigma_z^{\frac{\mu}{2}}(\mathcal E_h^{\Psi_i}-\mathcal E_h^{\widehat{\Psi}_i})\|_{\partial \mathcal T_h}\\
			&\le C(
			\|\sigma_z^{\frac{\mu}{2}}\mathcal E_h^{\bm{\Phi}_i}\|_{\mathcal T_h}+\|h_K^{-\frac 1 2}\sigma_z^{\frac{\mu}{2}}(\Pi_k^{\partial}\mathcal E_h^{\Psi_i}-\mathcal E_h^{\widehat{\Psi}_i})\|_{\partial \mathcal T_h}).	
		\end{split}
	\end{align}

	\begin{lemma}\label{Ih-stability} 
		Let $\mathcal I_h$ be the operator which was defined in \Cref{Oswarld_interpolation}, then we have 
		\begin{subequations}
			\begin{align}
				\|\sigma_{z}^{\frac{\mu}{2}-1}\mathcal I_h\mathcal E_h^{\Psi_1}\|_{\mathcal T_h} &\le C\|\sigma_{z}^{\frac{\mu}{2}-1}\mathcal E_h^{\Psi_1}\|_{\mathcal T_h},\label{Ih-stability_a}\\
				\|\sigma_{z}^{\frac{\mu}{2}}\nabla\mathcal I_h\mathcal E_h^{\Psi_1}\|_{\mathcal T_h}&\le C(
				\|\sigma_z^{\frac{\mu}{2}}\mathcal E_h^{\bm{\Phi}_i}\|_{\mathcal T_h}+h^{-\frac 1 2}\|\sigma_z^{\frac{\mu}{2}}(\Pi_k^{\partial}\mathcal E_h^{\Psi_i}-\mathcal E_h^{\widehat{\Psi}_i})\|_{\partial \mathcal T_h}).\label{Ih-stability_b}
			\end{align}
		\end{subequations}
	\end{lemma}
	\begin{proof}
		By the triangle inequality, \eqref{Oswarld_interpolation_b} and \eqref{weighed-inverse2}, we have
		\begin{align*}
			\|\sigma_{z}^{\frac{\mu}{2}-1}\mathcal I_h\mathcal E_h^{\Psi_1}\|_{\mathcal T_h}&\le 
			\|\sigma_{z}^{\frac{\mu}{2}-1}(\mathcal I_h\mathcal E_h^{\Psi_1}-\mathcal E_h^{\Psi_1})\|_{\mathcal T_h}
			+	\|\sigma_{z}^{\frac{\mu}{2}-1}\mathcal E_h^{\Psi_1}\|_{\mathcal T_h}\\
			&\le Ch^{\frac{1}{2}}\|\sigma_{z}^{\frac{\mu}{2}-1}[\![\mathcal E_h^{\Psi_1}]\!]\|_{\partial\mathcal T_h}+\|\sigma_{z}^{\frac{\mu}{2}-1}\mathcal E_h^{\Psi_1}\|_{\mathcal T_h}\\
			&\le C\|\sigma_{z}^{\frac{\mu}{2}-1}\mathcal E_h^{\Psi_1}\|_{\mathcal T_h},
		\end{align*}
		and this proves \eqref{Ih-stability_a}. Next, by the triangle inequality and \eqref{Oswarld_interpolation_b} we have
		\begin{align*}
			\|\sigma_{z}^{\frac{\mu}{2}}\nabla\mathcal I_h\mathcal E_h^{\Psi_1}\|_{\mathcal T_h}&\le 
			\|\sigma_{z}^{\frac{\mu}{2}}\nabla(\mathcal I_h\mathcal E_h^{\Psi_1}-\mathcal E_h^{\Psi_1})\|_{\mathcal T_h}
			+	\|\sigma_{z}^{\frac{\mu}{2}}\nabla\mathcal E_h^{\Psi_1}\|_{\mathcal T_h}\\
			&\le Ch^{-\frac 1 2}\|\sigma_{z}^{\frac{\mu}{2}}[\![\mathcal E_h^{\Psi_1}]\!]\|_{\mathcal T_h}+\|\sigma_{z}^{\frac{\mu}{2}}\nabla\mathcal E_h^{\Psi_1}\|_{\mathcal T_h}\\
			&=Ch^{-\frac 1 2}\|\sigma_{z}^{\frac{\mu}{2}}[\![\mathcal E_h^{\Psi_1}-\mathcal E_h^{\widehat{\Psi}_1}]\!]\|_{\partial\mathcal T_h}+\|\sigma_{z}^{\frac{\mu}{2}}\nabla\mathcal E_h^{\Psi_1}\|_{\mathcal T_h}.
		\end{align*}
		The last equality holds since $\mathcal E_h^{\widehat{\Psi}_1}$ is single valued on interior edges and $\mathcal E_h^{\widehat{\Psi}_1}=0$ on boundary edges. Next, we use \eqref{Stbilization_weighted} to get
		\begin{align*}	
			\|\sigma_{z}^{\frac{\mu}{2}}\nabla\mathcal I_h\mathcal E_h^{\Psi_1}\|_{\mathcal T_h}\le  C\left(
			\|\sigma_{z}^{\frac{\mu}{2}}\mathcal E_h^{\bm{\Phi}_1}\|_{\bm L^2(\mathcal T_h)}
			+\|{h_K^{-\frac 1 2}}\sigma_{z}^{\frac{\mu}{2}}(\Pi_k^{\partial}\mathcal E_h^{\Psi_1}-\mathcal E_h^{\widehat{\Psi}_1})\|_{\partial\mathcal T_h}\right).
		\end{align*}
	\end{proof}

	\begin{lemma}	
		If $\kappa$ is large enough, then we have:
		\begin{subequations}
			\begin{align}
				\|\sigma_{z}^{\frac{\mu}{2}-1}\mathcal E_h^{\Psi_1}\|_{\mathcal T_h} &\le  C \kappa^{-1}\left(\|\sigma_{z}^{\frac{\mu}{2}}\mathcal E_h^{\bm\Phi_1}\|_{\mathcal T_h}
				+ 	\|{ h_K^{-\frac 1 2}}\sigma_{z}^{\frac{\mu}{2}}(\Pi_k^{\partial}\mathcal E_h^{\Psi_1}-\mathcal E_h^{\widehat \Psi_1})\|_{\partial\mathcal T_h}+h^{1+\frac{\lambda}{2}}\right),\label{error_esti_Psi1}\\
				\|\sigma_{z}^{\frac{\mu}{2}-1}\mathcal E_h^{\Psi_2}\|_{\mathcal T_h}&\le   C\kappa^{-1}\left(\|\sigma_{z}^{\frac{\mu}{2}}\mathcal E_h^{\bm{\Phi}_2}\|_{\mathcal T_h}
				+\|{ h_K^{-\frac 1 2}}\sigma_{z}^{\frac{\mu}{2}}( \Pi_k^{\partial}\mathcal E_h^{\Psi_2}-\mathcal E_h^{\widehat{\Psi}_2})\|_{\partial\mathcal T_h}+h^{\frac \lambda 2}\right),\label{error_esti_Psi2}
			\end{align}	
			where $C$ is independent of $\kappa$ and $h$.
		\end{subequations}
	\end{lemma}
	\begin{proof}
		We use the Oswald operator $\mathcal I_h$ \eqref{Oswarld_interpolation} to split the following term into two terms:
		\begin{align*}
			\|\sigma_{z}^{\frac{\mu}{2}-1}\mathcal E_h^{\Psi_1}\|_{\mathcal T_h}^2=(\sigma_{z}^{\mu-2}\mathcal E_h^{\Psi_1},\mathcal E_h^{\Psi_1}-\mathcal I_h\mathcal E_h^{\Psi_1})_{\mathcal T_h}+(\sigma_{z}^{\mu-2}\mathcal E_h^{\Psi_1},\mathcal I_h\mathcal E_h^{\Psi_1})_{\mathcal T_h}:= I_1 + I_2.
		\end{align*}
		Next, we estimate the above two terms. First, by \eqref{Oswarld_interpolation_b} we have 
		\begin{align*}
			|I_1|  &\le Ch^{\frac 1 2}\|\sigma_{z}^{\frac{\mu}{2}-1}(\mathcal E_h^{\Psi_1}-\mathcal E_h^{\widehat \Psi_1})\|_{\partial\mathcal T_h}\|\sigma_{z}^{\frac{\mu}{2}-1}\mathcal E_h^{\Psi_1}\|_{\mathcal T_h}\\
			&\le C\kappa^{-1}h^{-\frac 1 2}\|\sigma_{z}^{\frac{\mu}{2}}(\mathcal E_h^{\Psi_1}-\mathcal E_h^{\widehat \Psi_1})\|_{\partial\mathcal T_h}\|\sigma_{z}^{\frac{\mu}{2}-1}\mathcal E_h^{\Psi_1}\|_{\mathcal T_h}.
		\end{align*}
		Here we used the fact that $\sigma_z\ge \kappa h$ (see \eqref{def_sigma}) and \eqref{pro_sigma_1}. By \eqref{Stbilization_weighted} we have
		\begin{align*}
			|I_1| &\le C\kappa^{-1}\left(\|\sigma_{z}^{\frac{\mu}{2}}\mathcal E_h^{\bm{\Phi}_1}\|_{\mathcal T_h}+h^{-\frac 1 2}\|\sigma_{z}^{\frac{\mu}{2}}(\Pi_k^{\partial}\mathcal E_h^{\Psi_1}-\mathcal E_h^{\widehat\Psi_1})\|_{\partial\mathcal T_h}\right)\|\sigma_{z}^{\frac{\mu}{2}-1}\mathcal E_h^{\Psi_1}\|_{\mathcal T_h}\\
			&\le \frac 1 4 \|\sigma_{z}^{\frac{\mu}{2}-1}\mathcal E_h^{\Psi_1}\|_{\mathcal T_h}^2 + C\kappa^{-2}\left(\|\sigma_{z}^{\frac{\mu}{2}}\mathcal E_h^{\bm{\Phi}_1}\|_{\mathcal T_h}^2+\|{ h_K^{-\frac 1 2}}\sigma_{z}^{\frac{\mu}{2}}(\Pi_k^{\partial}\mathcal E_h^{\Psi_1}-\mathcal E_h^{\widehat\Psi_1})\|_{\partial\mathcal T_h}^2\right).
		\end{align*}
		
		Next, for the term $I_2$,  by \eqref{pro_error_equa_Phi}, \eqref{commute} and \eqref{error_equa_Phi} we have 
		\begin{align*}
			I_2
			&=-\mathscr B(\bm\Pi_{k}^{o}\bm \Phi_{3},\Pi_{k+1}^{o}\Psi_3,\Pi_k^{\partial}\Psi_3;
			\mathcal E_h^{\bm{\Phi}_1},\mathcal E_h^{\Psi_1},\mathcal E_h^{\widehat{\Psi}_1})
			+\mathscr E(\bm\Phi_3,\Psi_3;	\mathcal E_h^{\bm{\Phi}_1},\mathcal E_h^{\Psi_1},\mathcal E_h^{\widehat{\Psi}_1})\\
			&=-\mathscr B(
			\mathcal E_h^{\bm{\Phi}_1},\mathcal E_h^{\Psi_1},\mathcal E_h^{\widehat{\Psi}_1};
			\bm\Pi_{k}^{o}\bm \Phi_{3},\Pi_{k+1}^{o}\Psi_3,\Pi_k^{\partial}\Psi_3)+\mathscr E(\bm\Phi_3,\Psi_3;	\mathcal E_h^{\bm{\Phi}_1},\mathcal E_h^{\Psi_1},\mathcal E_h^{\widehat{\Psi}_1})\\
			&=-\mathscr E(\bm\Phi_1,\Psi_1;	\bm\Pi_{k}^{o}\bm \Phi_{3},\Pi_{k+1}^{o}\Psi_3,\Pi_k^{\partial}\Psi_3) +\mathscr E(\bm\Phi_3,\Psi_3;	\mathcal E_h^{\bm{\Phi}_1},\mathcal E_h^{\Psi_1},\mathcal E_h^{\widehat{\Psi}_1})\\
			&=:I_{21}+I_{22}.
		\end{align*}
		For the term $I_{21}$ we have 
		\begin{align*}
			I_{21} &= -\mathscr E(\bm\Phi_1,\Psi_1;	\bm\Pi_{k}^{o}\bm \Phi_{3},\Pi_{k+1}^{o}\Psi_3,\Pi_k^{\partial}\Psi_3)\\
			& =\langle(\bm \Phi_1 -\bm \Pi_{k}^{o} \bm\Phi_1)\cdot\bm n,\Pi_{k+1}^{o}\Psi_3-\Pi_k^{\partial}\Psi_3\rangle_{\partial\mathcal T_h} +\langle h_K^{-1} (\Psi_1 -\Pi_{k+1}^o  \Psi_1   ),\Pi_k^\partial  (\Pi_{k+1}^o  \Psi_3 - \Psi_3) \rangle_{\partial\mathcal T_h}\\
			& = \langle(\bm \Phi_1 -\bm \Pi_{k}^{o} \bm\Phi_1)\cdot\bm n,\Pi_{k+1}^{o}\Psi_3-\Psi_3\rangle_{\partial\mathcal T_h} +\langle h_K^{-1} (\Psi_1 -\Pi_{k+1}^o  \Psi_1   ),\Pi_k^\partial  (\Pi_{k+1}^o  \Psi_3 - \Psi_3) \rangle_{\partial\mathcal T_h}.
		\end{align*}
		The last equality holds since $\langle \bm \Phi_1 \cdot\bm n, \Psi_3\rangle_{\partial\mathcal T_h} = 0 = \langle \bm \Phi_1 \cdot\bm n, \Pi_k^\partial\Psi_3\rangle_{\partial\mathcal T_h}$. Hence,
		\begin{align*}
			|I_{21}| &\le  |\langle \sigma_{z}^{\frac{\mu}{2}} (\bm \Phi_1 -\bm \Pi_{k}^{o} \bm\Phi_1)\cdot\bm n,\sigma_{z}^{-\frac{\mu}{2}}(\Pi_{k+1}^{o}\Psi_3-\Psi_3)\rangle_{\partial\mathcal T_h}|\\
			&\quad  +|\langle h_K^{-1} \sigma_{z}^{\frac{\mu}{2}} (\Psi_1 -\Pi_{k+1}^o  \Psi_1   ), \sigma_{z}^{\frac{\mu}{2}}\Pi_k^\partial  (\Pi_{k+1}^o  \Psi_3 - \Psi_3) \rangle_{\partial\mathcal T_h}|\\
			&\le Ch^2\|\sigma_{z}^{\frac{\mu}{2}}D^2\Psi_1\|_{L^2(\Omega)}\|\sigma_{z}^{-\frac{\mu}{2}}D^2\Psi_3\|_{L^2(\Omega)}\le Ch^{2+\frac{\lambda}{2}}\|\sigma_{z}^{-\frac{\mu}{2}}D^2\Psi_3\|_{L^2(\Omega)}.
		\end{align*}
		Similarly, for the term $I_{22}$ we have 
		\begin{align*}
			|I_{22}|\le Ch\|\sigma_{z}^{-\frac{\mu}{2}}D^2\Psi_3\|_{L^2(\Omega)}\left(\|\sigma_{z}^{\frac{\mu}{2}}\mathcal E_h^{\bm{\Phi}_1}\|_{\mathcal T_h}
			+	\|h_K^{-\frac 1 2}\sigma_{z}^{\frac{\mu}{2}}(\Pi_k^{\partial}\mathcal E_h^{\Psi_1}-\mathcal E_h^{\widehat\Psi_1})\|_{\partial\mathcal T_h}\right).
		\end{align*}
		This gives
		\begin{align}\label{proof2}
			|I_2|\le Ch\|\sigma_{z}^{-\frac{\mu}{2}}D^2\Psi_3\|_{L^2(\Omega)}\left(\|\sigma_{z}^{\frac{\mu}{2}}\mathcal E_h^{\bm{\Phi}_1}\|_{\mathcal T_h}
			+	\|h_K^{-\frac 1 2}\sigma_{z}^{\frac{\mu}{2}}(\Pi_k^{\partial}\mathcal E_h^{\Psi_1}-\mathcal E_h^{\widehat\Psi_1})\|_{\partial\mathcal T_h}+h^{1+\frac{\lambda}{2}}\right).
		\end{align}
		To estimate $\|\sigma_{z}^{-\frac{\mu}{2}}D^2\Psi_3\|_{L^2(\Omega)}$, we consider the dual problem: find $(\bm\Phi_3,\Psi_3)$ such that
		\begin{align*}
			\bm \Phi_3+\nabla \Psi_3= 0  \;\;\;\text{in}\ \Omega,\quad 
			\nabla\cdot\bm  \Phi_3= \sigma_{z}^{\mu-2}\mathcal I_h\mathcal E_h^{\Psi_1} \;\;\;\text{in}\ \Omega,\quad 
			\Psi_3 = 0                                           \;\;\;\text{on}\ \partial\Omega.
		\end{align*}
		Since $\sigma_{z}^{\mu-2}\mathcal I_h\mathcal E_h^{\Psi_1}\in H_0^1(\Omega)$,  by the regularity in \cite[Lemma 8.3.7]{Brenner_FEMBook_2008} and the estimates in \Cref{Ih-stability} we have 
		\begin{align}\label{dual_proof_1}
			\begin{split}
				\hspace{1em}&\hspace{-1em}\|\sigma_{z}^{-\frac{\mu}{2}}D^2\Psi_3\|_{\bm L^2(\Omega)}\le C\theta^{-1}\|\sigma_{z}^{2-\frac{\mu}{2}}D(\sigma_{z}^{\mu-2}\mathcal I_h\mathcal E_h^{\Psi_1})\|_{\bm L^2(\Omega)}\\
				&\le C\theta^{-1}\left(\|\sigma_{z}^{\frac{\mu}{2}-1}\nabla \sigma_z ~\mathcal I_h\mathcal E_h^{\Psi_1}\|_{\mathcal T_h}
				+	\|\sigma_{z}^{\frac{\mu}{2}}\nabla \mathcal I_h\mathcal E_h^{\Psi_1}\|_{\mathcal T_h}	\right)\\
				&\le C\theta^{-1}\left(\|\sigma_{z}^{\frac{\mu}{2}-1}\mathcal E_h^{\Psi_1}\|_{\mathcal T_h}+
				\|\sigma_{z}^{\frac{\mu}{2}}\mathcal E_h^{\bm\Phi_1}\|_{\mathcal T_h}+ h^{-\frac 1 2}\|\sigma_{z}^{\frac{\mu}{2}}(\Pi_k^{\partial}\mathcal E_h^{\Psi_1}-\mathcal E_h^{\widehat \Psi_1})\|_{\partial\mathcal T_h}\right).
			\end{split}
		\end{align}
		Substituting \eqref{dual_proof_1} into \eqref{proof2} we have 
		\begin{align*}
			|I_2|&\le Ch\theta^{-1}\left(\|\sigma_{z}^{\frac{\mu}{2}-1}\mathcal E_h^{\Psi_1}\|_{\mathcal T_h}+\|\sigma_{z}^{\frac{\mu}{2}}\mathcal E_h^{\bm\Phi_1}\|_{\mathcal T_h}
			+h^{-\frac 1 2}\|\sigma_{z}^{\frac{\mu}{2}}(\Pi_k^{\partial}\mathcal E_h^{\Psi_1}-\mathcal E_h^{\widehat \Psi_1})\|_{\partial\mathcal T_h}+h^{1+\frac{\lambda}{2}}\right)\\
			&\qquad\qquad \times \left(\|\sigma_{z}^{\frac{\mu}{2}}\mathcal E_h^{\bm{\Phi}_1}\|_{\mathcal T_h}+	\|\sigma_{z}^{\frac{\mu}{2}}(\Pi_k^{\partial}\mathcal E_h^{\Psi_1}-\mathcal E_h^{\widehat\Psi_1})\|_{\partial\mathcal T_h}+h^{1+\frac{\lambda}{2}}\right)\\
			&\le C\kappa^{-2}\left(\|\sigma_{z}^{\frac{\mu}{2}-1}\mathcal E_h^{\Psi_1}\|_{\mathcal T_h}^2+\|\sigma_{z}^{\frac{\mu}{2}}\mathcal E_h^{\bm\Phi_1}\|_{\mathcal T_h}^2
			+h^{-1}\|\sigma_{z}^{\frac{\mu}{2}}(\Pi_k^{\partial}\mathcal E_h^{\Psi_1}-\mathcal E_h^{\widehat \Psi_1})\|_{\partial\mathcal T_h}^2+h^{2+\lambda}\right).
		\end{align*}
		Then the desired result \eqref{error_esti_Psi1} follows  by summing $I_1$ and $I_2$, and taking $\kappa$ large enough. The proof of \eqref{error_esti_Psi2} is similar to the proof of \eqref{error_esti_Psi1}.
		
	\end{proof}

	\begin{lemma}\label{weight_err_thm}
		If $\kappa$ is large enough, then we have:
		\begin{subequations}
			\begin{align}
				\|\sigma_{z}^{\frac \mu 2}\mathcal E_h^{\bm\Phi_1}\|_{\mathcal T_h}+\|h_K^{-\frac 1 2}\sigma_{z}^{\frac \mu 2} (\Pi_k^{\partial}\mathcal E_h^{\Psi_1}- \mathcal E_h^{\widehat \Psi_1})\|_{L^2(\partial\mathcal T_h)} +\|h_K^{-\frac 1 2}\sigma_{z}^{\frac \mu 2} (\mathcal E_h^{\Psi_1}- \mathcal E_h^{\widehat \Psi_1})\|_{L^2(\partial\mathcal T_h)}
				\le  Ch^{1+\frac \lambda 2},\label{weight_err_1}\\
				\|\sigma_{z}^{\frac \mu 2}\mathcal E_h^{\bm\Phi_2}\|_{\mathcal T_h}+\|h_K^{-\frac 1 2}\sigma_{z}^{\frac \mu 2} (\Pi_k^{\partial}\mathcal E_h^{\Psi_2}- \mathcal E_h^{\widehat \Psi_2})\|_{L^2(\partial\mathcal T_h)} +\|h_K^{-\frac 1 2}\sigma_{z}^{\frac \mu 2} (\mathcal E_h^{\Psi_2}- \mathcal E_h^{\widehat \Psi_2})\|_{L^2(\partial\mathcal T_h)}
				\le  Ch^{\frac \lambda 2}.\label{weight_err_2}
			\end{align}	
		\end{subequations}
		
	\end{lemma}
	\begin{proof}
		On one hand, by the definition of $\mathscr B$ in \eqref{def_B}  we have:
		\begin{align*}
			\hspace{0.1em}&\hspace{-0.1em}\mathscr B(\mathcal E_h^{\bm \Phi_1}, \mathcal E_h^{\Psi_1}, \mathcal E_h^{\widehat \Psi_1}; \sigma_{z}^{\mu}\mathcal E_h^{\bm \Phi_1}, -\sigma_{z}^{\mu}\mathcal E_h^{\Psi_1}, -\sigma_{z}^{\mu}\mathcal E_h^{\widehat \Psi_1})\\
			&=(\mathcal E_h^{\bm  \Phi_1},  \sigma_{z}^{\mu}\mathcal E_h^{\bm  \Phi_1})_{{\mathcal{T}_h}}- ( \mathcal E_h^{\Psi_1}, \nabla\cdot(\sigma_{z}^{\mu}\mathcal E_h^{\bm \Phi_1}))_{{\mathcal{T}_h}}+\langle \mathcal E_h^{\widehat \Psi_1}, \sigma_{z}^{\mu}\mathcal E_h^{\bm \Phi_1}\cdot \bm{n} \rangle_{\partial{{\mathcal{T}_h}}}+ (\nabla\cdot\mathcal E_h^{\bm  \Phi_1},   \sigma_{z}^{\mu}\mathcal E_h^{\Psi_1})_{{\mathcal{T}_h}}\\
			&\;\;  - \langle \mathcal E_h^{\bm \Phi} \cdot\bm n,\sigma_{z}^{\mu}\mathcal E_h^{\widehat \Psi_1} \rangle_{\partial \mathcal T_h} +\langle { h_K^{-1}} (\Pi_k^{\partial}\mathcal E_h^{\Psi_1}- \mathcal E_h^{\widehat \Psi_1}),  \Pi_k^{\partial} (\sigma_{z}^{\mu}\mathcal E_h^{\Psi_1})-\sigma_{z}^{\mu}\mathcal E_h^{\widehat \Psi_1} \rangle_{\partial{{\mathcal{T}_h}}}\\
			&=(\mathcal E_h^{\bm \Phi_1},  \sigma_{z}^{\mu}\mathcal E_h^{\bm  \Phi_1})_{{\mathcal{T}_h}}- ( \mathcal E_h^{\Psi_1}, \sigma_{z}^{\mu} \nabla\cdot \mathcal E_h^{\bm  \Phi_1})_{{\mathcal{T}_h}} - ( \mathcal E_h^{\Psi_1}, \mu\sigma_{z}^{\mu-1} \nabla\sigma_z \cdot \mathcal E_h^{\bm  \Phi_1}  )_{{\mathcal{T}_h}} +\langle \mathcal E_h^{\widehat \Psi_1}, \sigma_{z}^{\mu}\mathcal E_h^{\bm \Phi_1}\cdot \bm{n} \rangle_{\partial{{\mathcal{T}_h}}}\\
			&\; + (\nabla\cdot\mathcal E_h^{\bm  \Phi_1},   \sigma_{z}^{\mu}\mathcal E_h^{\Psi_1})_{{\mathcal{T}_h}} - \langle \mathcal E_h^{\bm  \Phi_1} \cdot\bm n,\sigma_{z}^{\mu}\mathcal E_h^{\widehat \Psi_1} \rangle_{\partial \mathcal T_h}+\langle { h_K^{-1}} (\Pi_k^{\partial}\mathcal E_h^{\Psi_1}- \mathcal E_h^{\widehat \Psi_1}),  \Pi_k^{\partial} (\sigma_{z}^{\mu}\mathcal E_h^{\Psi_1})-\sigma_{z}^{\mu}\mathcal E_h^{\widehat \Psi_1} \rangle_{\partial{{\mathcal{T}_h}}}.
		\end{align*}
		This gives
		\begin{align}\label{equation_1}
			\begin{split}
				\hspace{1em}&\hspace{-1em} \mathscr B(\mathcal E_h^{\bm \Phi_1}, \mathcal E_h^{\Psi_1}, \mathcal E_h^{\widehat \Psi_1}; \sigma_{z}^{\mu}\mathcal E_h^{\bm \Phi_1}, -\sigma_{z}^{\mu}\mathcal E_h^{\Psi_1}, -\sigma_{z}^{\mu}\mathcal E_h^{\widehat \Psi_1})\\
				& =  (\mathcal E_h^{\bm\Phi_1},   \sigma_{z}^{\mu}\mathcal E_h^{\bm\Phi_1})_{{\mathcal{T}_h}}-  ( \mathcal E_h^{\Psi_1}, \mu\sigma_{z}^{\mu-1} \nabla\sigma_{z}\cdot\mathcal E_h^{\bm  \Phi_1})_{{\mathcal{T}_h}} + \|{ h_K^{-\frac 1 2}}\sigma_{z}^{\frac \mu 2}  (\Pi_k^{\partial}\mathcal E_h^{\Psi_1}- \mathcal E_h^{\widehat \Psi_1})\|_{L^2(\partial\mathcal T_h)}^2 \\
				&\quad +\langle { h_K^{-1}} (\Pi_k^{\partial}\mathcal E_h^{\Psi_1}- \mathcal E_h^{\widehat \Psi_1}),  \Pi_k^{\partial} (\sigma_{z}^{\mu}\mathcal E_h^{\Psi_1}) -\sigma_{z}^{\mu} \Pi_k^{\partial}\mathcal E_h^{\Psi_1}\rangle_{\partial{{\mathcal{T}_h}}}.
			\end{split}
		\end{align}
		
		On the other hand, by the error equation \eqref{error_equa_Phi} we get 
		\begin{align*}
			\hspace{0.1em}&\hspace{-0.1em}\mathscr B(\mathcal E_h^{\bm \Phi_1}, \mathcal E_h^{\Psi_1}, \mathcal E_h^{\widehat \Psi_1}; \sigma_{z}^{\mu}\mathcal E_h^{\bm \Phi_1}, -\sigma_{z}^{\mu}\mathcal E_h^{\Psi_1}, -\sigma_{z}^{\mu}\mathcal E_h^{\widehat \Psi_1})\\
			& = \mathscr B(\mathcal E_h^{\bm \Phi_1}, \mathcal E_h^{\Psi_1}, \mathcal E_h^{\widehat \Psi_1}; \sigma_{z}^{\mu}\mathcal E_h^{\bm  \Phi_1} - \bm \Pi_k^o (\sigma_{z}^{\mu}\mathcal E_h^{\bm \Phi_1}),  -\sigma_{z}^{\mu}\mathcal E_h^{\Psi_1}+ \Pi_{k+1}^o(\sigma_{z}^{\mu}\mathcal E_h^{\Psi_1}), -\sigma_{z}^{\mu}\mathcal E_h^{\widehat \Psi_1} + \Pi_k^\partial(\sigma_{z}^{\mu}\mathcal E_h^{\widehat \Psi_1}))\\
			&\quad  + \mathscr B(\mathcal E_h^{\bm  \Phi_1}, \mathcal E_h^{\Psi_1}, \mathcal E_h^{\widehat \Psi_1}; \bm \Pi_{k}^{o} (\sigma_{z}^{\mu}\mathcal E_h^{\bm \Phi_1}), {  -\Pi_{k+1}^{o}(\sigma_{z}^{\mu}\mathcal E_h^{\Psi_1})},- \Pi_k^\partial(\sigma_{z}^{\mu}\mathcal E_h^{\widehat \Psi_1}))\\
			& = \mathscr B(\mathcal E_h^{\bm \Phi_1}, \mathcal E_h^{\Psi_1}, \mathcal E_h^{\widehat \Psi_1}; \sigma_{z}^{\mu}\mathcal E_h^{\bm  \Phi_1} - \bm \Pi_k^o (\sigma_{z}^{\mu}\mathcal E_h^{\bm \Phi_1}),  -\sigma_{z}^{\mu}\mathcal E_h^{\Psi_1}+ \Pi_{k+1}^o(\sigma_{z}^{\mu}\mathcal E_h^{\Psi_1}), -\sigma_{z}^{\mu}\mathcal E_h^{\widehat \Psi_1} + \Pi_k^\partial(\sigma_{z}^{\mu}\mathcal E_h^{\widehat \Psi_1}))\\
			&\quad+\mathscr E(\bm{\Phi}_1,\Psi_1; \bm \Pi_{k}^{o} (\sigma_{z}^{\mu}\mathcal E_h^{\bm \Phi_1}), -\Pi_{k+1}^{o}(\sigma_{z}^{\mu}\mathcal E_h^{\Psi_1}),- \Pi_k^\partial(\sigma_{z}^{\mu}\mathcal E_h^{\widehat \Psi_1})).
		\end{align*}
		Next, we use the definition of $\mathscr B$ in \eqref{def_B} to get:
		\begin{align*}
			\hspace{1em}&\hspace{-1em} \mathscr B(\mathcal E_h^{\bm \Phi_1}, \mathcal E_h^{\Psi_1}, \mathcal E_h^{\widehat \Psi_1}; \sigma_{z}^{\mu}\mathcal E_h^{\bm  \Phi_1} - \bm \Pi_k^o (\sigma_{z}^{\mu}\mathcal E_h^{\bm \Phi_1}),  -\sigma_{z}^{\mu}\mathcal E_h^{\Psi_1}+ \Pi_{k+1}^o(\sigma_{z}^{\mu}\mathcal E_h^{\Psi_1}), -\sigma_{z}^{\mu}\mathcal E_h^{\widehat \Psi_1} + \Pi_k^\partial(\sigma_{z}^{\mu}\mathcal E_h^{\widehat \Psi_1}))\\
			&= \langle \mathcal E_h^{ \Psi_1} -  \mathcal E_h^{\widehat \Psi_1}, (\sigma_{z}^{\mu}\mathcal E_h^{\bm  \Phi_1} - \bm \Pi_{k}^{o}(\sigma_{z}^{\mu}\mathcal E_h^{\bm  \Phi_1}))\cdot \bm{n} \rangle_{\partial{{\mathcal{T}_h}}}\\
			&\quad  +\langle \Pi_k^\partial \mathcal E_h^{ \Psi_1}- \mathcal E_h^{\widehat \Psi_1}, { h_K^{-1}}(\Pi_{k+1}^o(\sigma_{z}^{\mu}\mathcal E_h^{\Psi_1}) -\sigma_{z}^{\mu}\mathcal E_h^{\Psi_1})\rangle_{\partial{{\mathcal{T}_h}}}.
		\end{align*}
		This implies
		\begin{align*}
			\hspace{1em}&\hspace{-1em} (\mathcal E_h^{\bm\Phi_1},   \sigma_{z}^{\mu}\mathcal E_h^{\bm\Phi_1})_{{\mathcal{T}_h}} + \|{ h_K^{-\frac 1 2}}\sigma_{z}^{\frac \mu 2}  (\Pi_k^{\partial}\mathcal E_h^{\Psi_1}- \mathcal E_h^{\widehat \Psi_1})\|_{L^2(\partial\mathcal T_h)}^2 \\
			&=  ( \mathcal E_h^{\Psi_1}, \mu\sigma_{z}^{\mu-1} \nabla\sigma_{z}\cdot\mathcal E_h^{\bm  \Phi_1})_{{\mathcal{T}_h}}-\langle { h_K^{-1}} (\Pi_k^{\partial}\mathcal E_h^{\Psi_1}- \mathcal E_h^{\widehat \Psi_1}),  \Pi_k^{\partial} (\sigma_{z}^{\mu}\mathcal E_h^{\Psi_1}) -\sigma_{z}^{\mu} \Pi_k^{\partial}\mathcal E_h^{\Psi_1}\rangle_{\partial{{\mathcal{T}_h}}}\\
			&\quad+\mathscr E(\bm{\Phi}_1,\Psi_1; \bm \Pi_{k}^{o} (\sigma_{z}^{\mu}\mathcal E_h^{\bm \Phi_1}), -\Pi_{k+1}^{o}(\sigma_{z}^{\mu}\mathcal E_h^{\Psi_1}),- \Pi_k^\partial(\sigma_{z}^{\mu}\mathcal E_h^{\widehat \Psi_1}))\\
			&\quad  +\langle \mathcal E_h^{ \Psi_1} -  \mathcal E_h^{\widehat \Psi_1}, (\sigma_{z}^{\mu}\mathcal E_h^{\bm  \Phi_1} - \bm \Pi_{k}^{o}(\sigma_{z}^{\mu}\mathcal E_h^{\bm  \Phi_1}){ )}\cdot \bm{n} \rangle_{\partial{{\mathcal{T}_h}}}\\
			&\quad  +\langle \Pi_k^\partial \mathcal E_h^{ \Psi_1}-  \mathcal E_h^{\widehat \Psi_1}, { h_K^{-1}}(\Pi_{k+1}^o(\sigma_{z}^{\mu}\mathcal E_h^{\Psi_1}) -\sigma_{z}^{\mu}\mathcal E_h^{\Psi_1})\rangle_{\partial{{\mathcal{T}_h}}}\\
			& {=:} I_1  + I_2 + I_3 + I_4+I_5.
		\end{align*}
		For the first term $I_1$, by \eqref{pro_sigma_2}, Young's inequality, \eqref{error_esti_Psi1} and taking $\kappa$ large enough (independent of $h$) we get
		\begin{align*}
			|I_1|\le C\|\sigma_z^{\frac \mu 2}\mathcal E_h^{\bm\Phi_1}\|_{\mathcal T_h}\|\sigma_z^{\frac \mu 2-1}\mathcal E_h^{\Psi_1}\|_{\mathcal T_h}
			\le \frac{1}{8}\left(\|\sigma_z^{\frac \mu 2}\mathcal E_h^{\bm\Phi_1}\|_{\mathcal T_h}^2 +\|{ h_K^{-\frac 1 2}}\sigma_{z}^{\frac{\mu}{2}}(\Pi_k^{\partial}\mathcal E_h^{\Psi_1}-\mathcal E_h^{\widehat{\Psi}_1})\|^2_{\partial\mathcal T_h}+h^{2+\lambda}\right).
		\end{align*}
		For the term $I_2$, let $\overline{\sigma_z^{\mu}} = \Pi_0^o \sigma_z^{\mu}$, then,  
		\begin{align*}
			|I_2| &= -\langle { h_K^{-1}} (\Pi_k^{\partial}\mathcal E_h^{\Psi_1}- \mathcal E_h^{\widehat \Psi_1}),   \sigma_{z}^{\mu}\mathcal E_h^{\Psi_1} -\sigma_{z}^{\mu} \Pi_k^{\partial}\mathcal E_h^{\Psi_1}\rangle_{\partial{{\mathcal{T}_h}}}\\
			& = -\langle { h_K^{-1}} (\Pi_k^{\partial}\mathcal E_h^{\Psi_1}- \mathcal E_h^{\widehat \Psi_1}),   \sigma_{z}^{\mu}(\mathcal E_h^{\Psi_1} - \Pi_k^{\partial}\mathcal E_h^{\Psi_1})\rangle_{\partial{{\mathcal{T}_h}}}\\
			& = -\langle { h_K^{-1}} (\Pi_k^{\partial}\mathcal E_h^{\Psi_1}- \mathcal E_h^{\widehat \Psi_1}),   (\sigma_{z}^{\mu}-\overline{\sigma_z^{\mu}})(\mathcal E_h^{\Psi_1} - \Pi_k^{\partial}\mathcal E_h^{\Psi_1})\rangle_{\partial{{\mathcal{T}_h}}}\\
			& \le C\sum_{K\in\mathcal T_h} { h_K^{-1}}\|\Pi_k^{\partial}\mathcal E_h^{\Psi_1}- \mathcal E_h^{\widehat \Psi_1}\|_{L^2(\partial K)} \|\mathcal E_h^{\Psi_1}\|_{L^2(\partial K)} \|\sigma_{z}^{\mu} -\overline{\sigma_z^{\mu}} \|_{L^{\infty}(\partial K)}.
		\end{align*}
		{	Here we used  the boundness of $\Pi_k^\partial$ on $L^2(\partial K)$}. By \eqref{Pro_jec_2},  \eqref{pro_sigma_1}, \eqref{error_esti_Psi1} and letting $\kappa$ be large enough,  we have
		\begin{align*}
			|I_2|&\le C\sum_{K\in\mathcal T_h} \|\Pi_k^{\partial}\mathcal E_h^{\Psi_1}- \mathcal E_h^{\widehat \Psi_1}\|_{L^2(\partial K)} \|\mathcal E_h^{\Psi_1}\|_{L^2(\partial K)} \|\sigma_{z}^{\mu-1}\|_{L^{\infty}(K)}\\
			&\le C\sum_{K\in\mathcal T_h} h^{-\frac 1 2}\|\sigma_{z}^{\frac{\mu}{2}}(\Pi_k^{\partial}\mathcal E_h^{\Psi_1}- \mathcal E_h^{\widehat \Psi_1})\|_{L^2(\partial K)} \|\sigma_{z}^{\frac{\mu}{2}-1}\mathcal E_h^{\Psi_1}\|_{L^2(K)}\\
			&\le \frac{1}{8}\left(\|\sigma_z^{\frac \mu 2}\mathcal E_h^{\bm\Phi_1}\|_{\mathcal T_h}^2 +h^{-1}\|\sigma_{z}^{\frac{\mu}{2}}(\Pi_k^{\partial}\mathcal E_h^{\Psi_1}-\mathcal E_h^{\widehat{\Psi}_1})\|^2_{\partial\mathcal T_h}+h^{2+\lambda}\right).
		\end{align*}
		Using the same technique that we used to estimate $I_2$,  for the term $I_3$ we have
		\begin{align*}
			|I_3|\le \frac{1}{8}\left(\|\sigma_z^{\frac \mu 2}\mathcal E_h^{\bm\Phi_1}\|_{\mathcal T_h}^2 +h^{-1}\|\sigma_{z}^{\frac{\mu}{2}}(\Pi_k^{\partial}\mathcal E_h^{\Psi_1}-\mathcal E_h^{\widehat{\Psi}_1})\|^2_{\partial\mathcal T_h}+h^{2+\lambda}\right).
		\end{align*}
		For the term $I_4$, we use \eqref{Stbilization_weighted} and \eqref{weight-app}  to get
		\begin{align*}
			|I_4|&\le  \|\sigma_{z}^{\frac{\mu}{2}}(\mathcal E_h^{\Psi_1}-\mathcal E_h^{\widehat\Psi_1})\|_{\partial \mathcal T_h}
			\|\sigma_{z}^{-\frac{\mu}{2}}(\sigma_{z}^{\mu}\mathcal E_h^{\bm  \Phi_1} - \bm \Pi_{k}^{o}(\sigma_{z}^{\mu}\mathcal E_h^{\bm \Phi_1}))\|_{\partial \mathcal T_h}\\
			&\le Ch\left(\|\sigma_z^{\frac{\mu}{2}}\mathcal E_h^{\bm{\Phi}_1}\|_{\mathcal T_h}	+ h^{-\frac 1 2}\|\sigma_z^{\frac{\mu}{2}}(\Pi_k^{\partial}\mathcal E_h^{\Psi_1}-\mathcal E_h^{\widehat{\Psi}_1})\|_{\partial\mathcal T_h}\right)\|\sigma_{z}^{\frac{\mu}{2}-1}\mathcal E_h^{\bm{\Phi}_1}\|_{\mathcal T_h}.
		\end{align*}
		{
			
			By the definition of $\sigma_z$ in \eqref{def_sigma} we have $\sigma_z\ge \kappa h$,  and choosing $\kappa$ big enough we obtain
		}
		\begin{align*}
			|I_4|
			&\le  \frac{1}{8}\left(\|\sigma_{z}^{\frac{\mu}{2}}\mathcal E_h^{\bm{\Phi}_1}\|_{\mathcal T_h}^2+
			h^{-1}\|\sigma_{z}^{\frac{\mu}{2}}(\Pi_k^{\partial}\mathcal E_h^{\Psi_1}-\mathcal E_h^{\widehat \Psi_1})\|_{\partial\mathcal T_h}^2\right).
		\end{align*}
		For the term $I_5$, { we use \eqref{weight-app}  to get}
		\begin{align*}
			|I_5|&\le  h^{-1}\|\sigma_{z}^{\frac{\mu}{2}}(\Pi_k^\partial\mathcal E_h^{\Psi_1}-\mathcal E_h^{\widehat\Psi_1})\|_{\partial \mathcal T_h}
			\|\sigma_{z}^{-\frac{\mu}{2}}(\Pi_{k+1}^o(\sigma_{z}^{\mu}\mathcal E_h^{\Psi_1}) -\sigma_{z}^{\mu}\mathcal E_h^{\Psi_1})\|_{\partial \mathcal T_h}\\
			&\le Ch^{-\frac 1 2}\|\sigma_{z}^{\frac{\mu}{2}}(\Pi_k^\partial\mathcal E_h^{\Psi_1}-\mathcal E_h^{\widehat\Psi_1})\|_{\partial \mathcal T_h}\|\sigma_{z}^{\frac{\mu}{2}-1}\mathcal E_h^{{\Psi}_1}\|_{\mathcal T_h}.
		\end{align*}
		Next, we apply \eqref{error_esti_Psi1} and take $\kappa$ sufficiently  large to have
		\begin{align*}
			|I_5| \le  \frac{1}{8}\left(\|\sigma_{z}^{\frac{\mu}{2}}\mathcal E_h^{\bm{\Phi}_1}\|_{\mathcal T_h}^2+
			h^{-1}\|\sigma_{z}^{\frac{\mu}{2}}(\Pi_k^{\partial}\mathcal E_h^{\Psi_1}-\mathcal E_h^{\widehat \Psi_1})\|_{\partial\mathcal T_h}^2+h^{2+\lambda}\right).
		\end{align*}
		Summing the estimates for $\{|I_k|\}_{k=1}^5$ gives 
		\begin{align*}
			\|\sigma_{z}^{\frac \mu 2}\mathcal E_h^{\bm\Phi_1}\|_{\mathcal T_h}
			+
			\|h_K^{-\frac 1 2}\sigma_{z}^{\frac \mu 2} (\Pi_k^{\partial}\mathcal E_h^{\Psi_1}- \mathcal E_h^{\widehat \Psi_1})\|_{L^2(\partial\mathcal T_h)} \le  Ch^{1+\frac \lambda 2}.
		\end{align*}
		We get the desired result  \eqref{weight_err_1} by using the above estimate and \eqref{Stbilization_weighted}. The proof of \eqref{weight_err_2} is similar to the proof of \eqref{weight_err_1}.
	\end{proof}

	\paragraph*{Step 3: Proof of  \eqref{maxmimu_norm_u}-\eqref{maxmimu_norm_q}  in \Cref{main_result_Linfty_norm}}	

	\begin{proof}
		We choose $\delta_{1,z}$ so that $\|u_h\|_{ L^\infty(\Omega)} = (\delta_{1,z}, u_h)_{\mathcal T_h}$, then  using \Cref{eq_B}, 
		\begin{align*}
			-(\delta_{1,z}, u_h)_{\mathcal T_h} & = \mathscr B(\bm \Phi_{1,h}, \Psi_{1,h}, \widehat{\Psi}_{1,h}; \bm q_h, u_h,\widehat u_h)&\textup{by }\eqref{Green_HDG_1}\\
			& =  \mathscr B( \bm q_h, u_h,\widehat u_h; \bm \Phi_{1,h}, \Psi_{1,h}, \widehat{\Psi}_{1,h})&\textup{by }\eqref{commute}\\
			& =  \mathscr B( \bm q, u, u; \bm \Phi_{1,h}, \Psi_{1,h}, \widehat{\Psi}_{1,h})
			&\textup{by }\eqref{HDG_exact}\\
			& = \mathscr B(\bm \Phi_{1,h}, \Psi_{1,h}, \widehat{\Psi}_{1,h}; \bm q, u, u)&\textup{by }\eqref{commute}\\
			&  =\mathscr B(\bm \Phi_{1,h} - \bm \Phi_{1}, \Psi_{1,h}-\Psi_{1}, \widehat{\Psi}_{1,h} - {\Psi}_{1}; \bm q, u, u) \\
			&\qquad + \mathscr B(\bm \Phi_{1}, \Psi_{1}, {\Psi}_{1}; \bm q, u, u)\\
			&  =\mathscr B(\bm \Phi_{1,h} - \bm \Phi_{1}, \Psi_{1,h}-\Psi_{1}, \widehat{\Psi}_{1,h} - {\Psi}_{1}; \bm q, u, u) - (\delta_{1,z}, u)_{\mathcal T_h}
			&\textup{by }\eqref{Green_HDG_1}.
		\end{align*}
		By the definition of $ \mathscr B$ in  \eqref{def_B} we have 
		\begin{align*}
			\hspace{1em}&\hspace{-1em} \mathscr B(\bm \Phi_{1,h} - \bm \Phi_{1}, \Psi_{1,h}-\Psi_{1}, \widehat{\Psi}_{1,h} - {\Psi}_{1}; \bm q, u, u)\\
			& = (\bm \Phi_{1,h} - \bm \Phi_{1}, \bm q)_{\mathcal T_h}  - ( \Psi_{1,h}-\Psi_{1},\nabla\cdot \bm q)_{\mathcal T_h}  + \langle\widehat{\Psi}_{1,h} - {\Psi}_{1},\bm q\cdot\bm n \rangle_{\partial\mathcal T_h}\\
			&\quad  -(\nabla\cdot (\bm \Phi_{1,h} - \bm \Phi_{1}), u)_{\mathcal T_h} +  \langle(\bm \Phi_{1,h} - \bm \Phi_{1})\cdot \bm n, u\rangle_{\partial\mathcal T_h}\\
			& = ( \nabla(\Psi_{1,h}-\Psi_{1}),\bm q)_{\mathcal T_h}
			- \langle \Psi_{1,h}-\widehat\Psi_{1,h},\bm q\cdot\bm n\rangle_{\partial\mathcal T_h},
		\end{align*}
		where we used integration by parts  and the fact that $ \langle\widehat{\Psi}_{1,h},\bm q\cdot\bm n \rangle_{\partial\mathcal T_h} = 0$ and  $ \langle {\Psi}_{1},\bm q\cdot\bm n \rangle_{\partial\mathcal T_h} = 0$ in the last equality. Therefore,
		\begin{align*}
			(\delta_{1,z}, u_h)_{\mathcal T_h}  &= (\delta_{1,z}, u)_{\mathcal T_h}-( \nabla(\Psi_{1,h}-\Psi_{1}),\bm q)_{\mathcal T_h}
			+ \langle \Psi_{1,h}-\widehat\Psi_{1,h},\bm q\cdot\bm n\rangle_{\partial\mathcal T_h} \\
			&=  (\delta_{1,z}, u)_{\mathcal T_h} + T_1 + T_2.
		\end{align*}
		Next, we estimate the above two terms $T_1$ and $T_2$. For the term $T_1$ we have 
		\begin{align*}
			|T_1|&\le\|\sigma_z^{\frac{\mu}{2}}\nabla(\Psi_{1,h}-\Psi_1)\|_{L^2(\Omega)}\|\sigma_z^{-\frac \mu 2}\|_{L^2(\Omega)} \|\bm q\|_{L^{\infty}(\Omega)} \\
			&\le (\|\sigma_z^{\frac{\mu}{2}}\nabla(\Psi_{1,h}-\Pi_{k+1}^o\Psi_1)\|_{L^2(\Omega)} + \|\sigma_z^{\frac{\mu}{2}}\nabla(\Pi_{k+1}^o\Psi_1-\Psi_1)\|_{L^2(\Omega)})\|\sigma_z^{-\frac \mu 2}\|_{L^2(\Omega)} \|\bm q\|_{L^{\infty}(\Omega)}\\
			& \le Ch\|\bm q\|_{L^{\infty}(\Omega)}.
		\end{align*}
		Here we used \eqref{Stbilization_weighted}, \eqref{weight_err_1}, \eqref{L2_prijection_sigma} and\eqref{sigma_lemma_2_constant}. For the term $T_2$ we have 
		\begin{align*}
			|T_2|&=| \langle\mathcal E_h^{\Psi_1}-\mathcal E_h^{\widehat\Psi_1}+\Pi_{k+1}^{o}\Psi_1-\Pi_k^{\partial}\Psi_1,\bm q\cdot\bm n\rangle_{\partial\mathcal T_h}|=| \langle\mathcal E_h^{\Psi_1}-\mathcal E_h^{\widehat\Psi_1}+\Pi_{k+1}^{o}\Psi_1-\Psi_1,\bm q\cdot\bm n\rangle_{\partial\mathcal T_h}|\\
			&\le\|\sigma_z^{\frac{\mu}{2}}(\mathcal E_h^{\Psi_1}-\mathcal E_h^{\widehat\Psi_1}+\Pi_{k+1}^{o}\Psi_1-\Psi_1)\|_{L^2(\partial\mathcal T_h)}\|\sigma_z^{-\frac \mu 2}\|_{L^2(\partial\mathcal T_h)}
			\|\bm q\|_{L^{\infty}(\Omega)}\\
			&\le Ch\|\bm q\|_{L^{\infty}(\Omega)},
		\end{align*}
		where we used \eqref{weight_err_1}, \eqref{L2_prijection_sigma} and \eqref{sigma-bd}. Combining the above two estimates,  and using the fact that {$\|\delta_{1,z}\|_{L^1(\Omega)}\le C$ (see \eqref{Bound_L1})  give
			\begin{align*}
				\|u_h\|_{L^\infty(\Omega)} 
				&\le
				C(\|u\|_{L^\infty(\Omega)}+h\|\bm q\|_{\bm L^{\infty}(\Omega)}).
		\end{align*}}
		This completes the proof of \eqref{maxmimu_norm_u}. The proof of \eqref{maxmimu_norm_q} is similar to the proof of \eqref{maxmimu_norm_u}.
	\end{proof}

	\paragraph*{Step 4: Proof of \eqref{thmL_infty_q}-\eqref{thmL_infty_u}  in \Cref{main_result_Linfty_norm}}		
	In this step, we only prove \eqref{thmL_infty_u} since the proof of \eqref{thmL_infty_q} is very similar. We take $(\bm v_h, w_h, \widehat w_h) = (\bm\Pi_{k}^{o} \bm q - \bm q_h, \Pi_{k+1}^{o} u - u_h, \Pi_k^\partial  u-\widehat u_h)$ in \eqref{Green_HDG_1} to get
	\begin{align*}
		\hspace{1em}&\hspace{-1em}	- (\delta_{1,z}, \Pi_{k+1}^{o} u - u_h)_{\mathcal T_h} \\
		& =  \mathscr B(\bm \Phi_{1,h},\Psi_{1,h},\widehat \Psi_{1,h}; \bm\Pi_{k}^{o} \bm q - \bm q_h, \Pi_{k+1}^{o} u - u_h, \Pi_k^\partial  u-\widehat u_h)\\
		& =  \mathscr B( \bm\Pi_{k}^{o} \bm q - \bm q_h, \Pi_{k+1}^{o} u - u_h, \Pi_k^\partial  u-\widehat u_h; \bm \Phi_{1,h},\Psi_{1,h},\widehat \Psi_{1,h}) & \textup{by } \eqref{commute}\\
		& =  \mathscr B( \bm\Pi_{k}^{o} \bm q - \bm q, \Pi_{k+1}^{o} u - u, \Pi_k^\partial  u-u; \bm \Phi_{1,h},\Psi_{1,h},\widehat \Psi_{1,h})& \textup{by } \eqref{HDG_exact}\\
		& =  \mathscr B(\bm \Phi_{1,h},\Psi_{1,h},\widehat \Psi_{1,h}; \bm\Pi_{k}^{o} \bm q - \bm q, \Pi_{k+1}^{o} u - u, \Pi_k^\partial  u- u ) & \textup{by } \eqref{commute}\\
		& =- (\Psi_{1,h}, \nabla\cdot (\bm\Pi_{k}^{o} \bm q - \bm q))_{{\mathcal{T}_h}}+\langle \widehat \Psi_{1,h}, (\bm\Pi_{k}^{o} \bm q - \bm q)\cdot \bm{n} \rangle_{\partial{{\mathcal{T}_h}}}\\
		&\quad-\langle h_K^{-1}(\Pi_k^{\partial}\Psi_{1,h}-\widehat\Psi_{1,h}), \Pi_{k+1}^{o}u-u \rangle_{\partial\mathcal T_h},
	\end{align*}
	where we used  the definition of $ \mathscr B$ in the last step. Next, integration by part gives	
	\begin{align*}
		(\delta_{1,z}, \Pi_{k+1}^{o} u - u_h)_{\mathcal T_h} &= \langle\Psi_{1,h}- \widehat \Psi_{1,h}, (\bm\Pi_{k}^{o} \bm q - \bm q)\cdot \bm{n} \rangle_{\partial{{\mathcal{T}_h}}}+\langle {h_K^{-1}}(\Pi_k^{\partial}\Psi_{1,h}-\widehat\Psi_{1,h}), \Pi_{k+1}^{o}u-u \rangle_{\partial\mathcal T_h}\\
		& = T_1 + T_2.
	\end{align*}
	We now estimate the above two terms. For the term $T_1$, by the triangle  inequality we have 
	\begin{align*}
		|T_1|&= |\langle\Psi_{1,h}-\Pi_{k+1}^o\Psi_1 +\Pi_{k+1}^o\Psi_1 - \Pi_k^\partial \Psi_1 + \Pi_k^\partial \Psi_1-\widehat \Psi_{1,h}, (\bm\Pi_{k}^{o} \bm q - \bm q)\cdot \bm{n} \rangle_{\partial{{\mathcal{T}_h}}}|\\
		&\le  |\langle\mathcal E_h^{\Psi_1}- \mathcal E_h^{\widehat \Psi_1}, (\bm\Pi_{k}^{o} \bm q - \bm q)\cdot \bm{n} \rangle_{\partial{{\mathcal{T}_h}}}| +  |\langle\Pi_{k+1}^o\Psi_1 - \Pi_k^\partial \Psi_1 , (\bm\Pi_{k}^{o} \bm q - \bm q)\cdot \bm{n} \rangle_{\partial{{\mathcal{T}_h}}}| \\
		&=  |\langle\mathcal E_h^{\Psi_1}- \mathcal E_h^{\widehat \Psi_1}, (\bm\Pi_{k}^{o} \bm q - \bm q)\cdot \bm{n} \rangle_{\partial{{\mathcal{T}_h}}}| +  |\langle\Pi_{k+1}^o\Psi_1 -  \Psi_1, (\bm\Pi_{k}^{o} \bm q - \bm q)\cdot \bm{n} \rangle_{\partial{{\mathcal{T}_h}}}|.
	\end{align*}
	Here we used the fact that $ \langle \Pi_k^\partial\Psi_1,\bm q\cdot\bm n \rangle_{\partial\mathcal T_h} = 0$ and  $ \langle {\Psi}_{1},\bm q\cdot\bm n \rangle_{\partial\mathcal T_h} = 0$ in  the last equality. Next, by the Cauchy-Schwarz  and  H\"older inequality we have 
	\begin{align*}
		|T_1|&\le \|\sigma_z^{-\frac{\mu}{2}}\|_{L^2(\partial\mathcal T_h)}(\|\sigma_z^{\frac{\mu}{2}}(\mathcal E_h^{\Psi_1}- \mathcal E_h^{\widehat \Psi_1})\|_{L^2(\partial\mathcal T_h)}+\|\sigma_z^{\frac{\mu}{2}}(\Pi_{k+1}^o\Psi_1 -  \Psi_1)\|_{L^2(\partial\mathcal T_h)})\|\bm {\Pi}_k^o\bm q-\bm q\|_{\bm L^\infty(\Omega)}          \\
		&\le C\theta^{-\frac{1}{2}}\|\sigma_z^{-\frac{\mu}{2}}\|_{L^2(\Omega)}(\|\sigma_z^{\frac{\mu}{2}}(\mathcal E_h^{\Psi_1}- \mathcal E_h^{\widehat \Psi_1})\|_{\partial\mathcal T_h}+\|\sigma_z^{\frac{\mu}{2}}(\Pi_{k+1}^o\Psi_1 -  \Psi_1)\|_{\partial\mathcal T_h}\|\bm {\Pi}_k^o\bm q-\bm q\|_{\bm L^\infty(\Omega)},
	\end{align*}
	where we used \eqref{sigma-bd} in the last inequality. Next, we use \eqref{sigma_lemma_2_constant}, \eqref{L2_prijection_sigma}, \eqref{regularity_psi_1}, \eqref{weight_err_1} and \eqref{Pro_jec_1} to get
	\begin{align*}
		|T_1|\le C\theta^{-\frac{1}{2}} h^{-\frac {\lambda}{2}} h^{\frac 3 2+\frac{\lambda}{2}} h^{k+1}|\bm q|_{\bm W^{k+1,\infty}(\Omega)}\le C h^{k+2}|\bm q|_{\bm W^{k+1,\infty}(\Omega)}.
	\end{align*}
	For the term $T_2$, we use the same arguments of $T_1$  to get
	\begin{align*}
		|T_2|\le C h^{k+2}|u|_{W^{k+2,\infty}(\Omega)}.
	\end{align*}
	This implies that
	\begin{align*}
		\|\Pi_{k+1}^{o} u - u_h\|_{L^\infty(\Omega)}\le C h^{k+2}(|u|_{W^{k+2,\infty}(\Omega)}+|\bm q|_{\bm W^{k+1,\infty}(\Omega)}).
	\end{align*}
	Using the above inequality and \eqref{Pro_jec_1} we have our desired result.

	\section{Numerical Experiments}
	\label{NumericalExperiments}
	
	In this section, we present two examples to illustrate our theoretical results.
	
	\begin{example}\label{example_1}
		We  first  test the convergence rate of the $L^\infty$ norm estimate in 2D, in order to provide an example on a non simplicial mesh. The domain $\Omega = (0,1)\times (0,1)$ and we partition by ladder-shaped meshes; see \Cref{fig2}. 
		The exact solution $u(x,y)$ is chosen to be $\sin^2(\pi x)\sin^2(\pi y)$. The source term $f$ is  chosen to match the exact solution of \eqref{Poisson} and the approximation errors  are listed in \Cref{table_1}. The rates  match the theoretical predictions in \Cref{main_result_Linfty_norm}.
		\begin{figure}[H]
			\centering
			\includegraphics[scale=0.3]{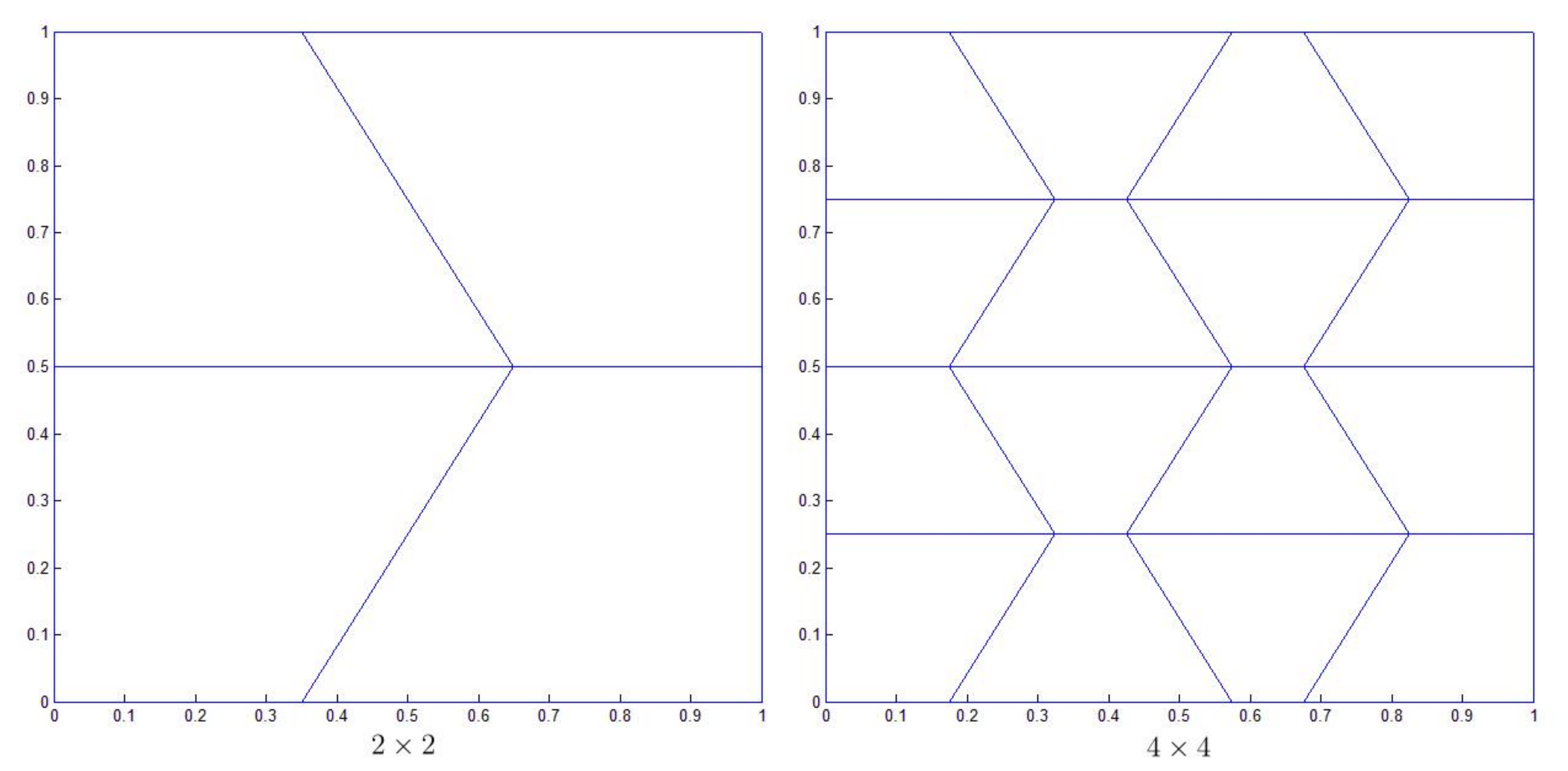}
			\caption{\label{fig2} Ladder-shaped meshes for $\Omega = (0,1)\times (0,1)$.}
		\end{figure}
		\begin{table}
			{
				\begin{tabular}{c|c|c|c|c|c|c|c}
					\Xhline{1pt}

					\multirow{2}{*}{}
					&\multirow{2}{*}{$h$}	
					&\multicolumn{2}{c|}{$k=0$}	
					&\multicolumn{2}{c|}{$k=1$}	
					&\multicolumn{2}{c}{$k=2$}	\\
					\cline{3-8}
					& &Error &Rate
					&Error &Rate
					&Error &Rate
					\\
					\cline{1-8}
					\multirow{5}{*}{ $\|\bm{q}-\bm{q}_h\|_{\bm L^\infty(\Omega)}$}
					&	$2^{-4}$	&	 1.95E-01	&	-	    &	   2.84E-02	&	-	    &	  1.23E-03	&	-	 \\ 
					&	$2^{-5}$	&	 9.38E-02	&	1.06	&	   7.72E-02	&	1.88	&	  1.30E-04	&	3.24	 \\ 
					&	$2^{-6}$	&	 4.68E-02	&	1.00	&	   1.98E-02	&	1.96	&	  2.12E-05	&	2.62	 \\ 
					&	$2^{-7}$	&	 2.33E-02	&	1.01	&	   4.95E-02	&	2.00	&	  2.65E-06	&	3.00	 \\ 
					&	$2^{-8}$	&	 1.16E-02	&	1.01	&	   1.23E-02	&	2.00	&	  3.33E-07	&	2.99	 \\

					\cline{1-8}
					\multirow{5}{*}{$\|u-u_h\|_{L^\infty(\Omega)}$}
					&	$2^{-4}$	&	 4.90E-02	&	-	    &	   4.96E-04	&	-	    &	  1.48E-04	&	-	 \\ 
					&	$2^{-5}$	&	 1.41E-02	&	1.80	&	   5.79E-05	&	3.10	&	  1.22E-05	&	3.60	 \\ 
					&	$2^{-6}$	&	 3.54E-03	&	1.99	&	   6.96E-06	&	3.06	&	  7.53E-07	&	4.02	 \\ 
					&	$2^{-7}$	&	 8.88E-04	&	2.00	&	   8.68E-07	&	3.00	&	  4.80E-08	&	3.97	 \\ 
					&	$2^{-8}$	&	 2.22E-04	&	2.00	&	   1.08E-07	&	3.01	&	  3.13E-09	&	3.94	 \\

					\Xhline{1pt}

				\end{tabular}
			}
			\caption{\Cref{example_1}: $L^\infty(\Omega)$ errors for $\bm{q}_h$ and $u_h$ on domain $(0,1)\times (0,1)$ with  Ladder-shaped meshes.}\label{table_1}
		\end{table}

	\end{example}

	\begin{example}\label{example_2}
		Next, we test the convergence rate of the $L^\infty$ norm estimate in 3D. The domain $\Omega = (0,1)\times (0,1)\times (0,1)$ and we use uniform  simplex meshes. 
		The exact solution $u(x,y,z)$ is chosen to be $u(x,y,z) = x(x-1)y(y-1)z(z-1)$. The source term $f$ is  chosen to match the exact solution of \eqref{Poisson} and the approximation errors  are listed in \Cref{tab2}. The rates  match the theoretical predictions in \Cref{main_result_Linfty_norm}.

		\begin{table}[H]
			{
				\begin{tabular}{c|c|c|c|c|c|c|c}
					\Xhline{1pt}

					\multirow{2}{*}{}
					&\multirow{2}{*}{$h$}	
					&\multicolumn{2}{c|}{$k=0$}	
					&\multicolumn{2}{c|}{$k=1$}	
					&\multicolumn{2}{c}{$k=2$}	\\
					\cline{3-8}
					& &Error &Rate
					&Error &Rate
					&Error &Rate
					\\
					\cline{1-8}
					\multirow{5}{*}{ $\|\bm{q}-\bm{q}_h\|_{\bm L^\infty(\Omega)}$}
					&	$2^{-1}$	&	 9.79E-03	&	-	    &	   6.47E-03	&	-	    &	  7.58E-03	&	-	 \\ 
					&	$2^{-2}$	&	 8.61E-03	&	0.18	&	   1.84E-03	&	1.81	&	  1.49E-04	&	2.35	 \\ 
					&	$2^{-3}$	&	 5.20E-03	&	0.73	&	   5.01E-04	&	1.88	&	  2.78E-05	&	2.42	 \\ 
					&	$2^{-4}$	&	 2.77E-03	&	0.91	&	   1.30E-04	&	1.94	&	  4.18E-06	&	2.73	 \\ 
					&	$2^{-5}$	&	 1.42E-03	&	0.96	&	   3.36E-05	&	1.95	&	  5.69E-07	&	2.88	 \\

					\cline{1-8}
					\multirow{5}{*}{$\|u-u_h\|_{L^\infty(\Omega)}$}
					&	$2^{-1}$	&	 4.51E-03	&	-	    &	   9.86E-04	&	-	    &	  2.64E-04	&	-	 \\ 
					&	$2^{-2}$	&	 1.31E-03	&	1.78	&	   1.31E-04	&	2.90	&	  1.65E-05	&	4.00	 \\ 
					&	$2^{-3}$	&	 3.38E-04	&	1.96	&	   1.73E-05	&	2.93	&	  1.03E-07	&	4.00	 \\ 
					&	$2^{-4}$	&	 8.51E-05	&	1.99	&	   2.23E-06	&	2.96	&	  6.44E-08	&	4.00	 \\ 
					&	$2^{-5}$	&	 2.13E-05	&	2.00	&	   2.92E-07	&	2.93	&	  4.03E-09	&	4.00	 \\

					\Xhline{1pt}

				\end{tabular}
			}
			\caption{\Cref{example_2}: $L^\infty(\Omega)$ errors for $\bm{q}_h$ and $u_h$ on domain $(0,1)\times (0,1)\times (0,1)$ with  simplex meshes.} \label{tab2}
		\end{table}

	\end{example}

	\section{Conclusion}
	We have proved  sharp $L^\infty$ norm estimates for the Poisson equation in both 2D and 3D. In the future,  we would like to extend the results to some other models, such as the Stokes equation and Maxwell's equations.

\bibliographystyle{siamplain}
\bibliography{Maxwell,Linfity,Dirichlet_Boundary_Control,Mypapers,Added,HDG}

\end{document}